\documentclass[a4paper, 11pt]{amsart}

\usepackage{amsmath, amsthm, amssymb}
\usepackage{enumitem}
\usepackage{xcolor}
\usepackage{hyperref}
\hypersetup{
    colorlinks,
    linkcolor={red!50!black},
    citecolor={blue!50!black},
    urlcolor={blue!80!black}
}
\usepackage{geometry} 
\usepackage{tikz-cd}
\usepackage{stackengine}

\usepackage{overpic}
\usepackage[ansinew]{inputenc}
\usepackage{graphicx}
\usepackage[rightcaption]{sidecap}
\usepackage{caption}
\usepackage{subcaption}

\newtheorem{theorem}{Theorem}[section]
\newtheorem{corollary}[theorem]{Corollary}
\newtheorem{lemma}[theorem]{Lemma}
\newtheorem{proposition}[theorem]{Proposition}

\theoremstyle{definition}
\newtheorem{definition}[theorem]{Definition}
\newtheorem{remark}[theorem]{Remark}

\makeatletter
\newcommand{\setword}[2]{%
  \phantomsection
  #1\def\@currentlabel{\unexpanded{#1}}\label{#2}%
}
\makeatother

\newcommand{\hi}{\mbox{$h_{\infty}$}}
\newcommand{\hE}{\mbox{$h_{E}$}}
\newcommand{\hiE}{\mbox{$h_{E}^i$}}
\newcommand{\cU}{\mbox{${\mathcal U}$}}
\newcommand{\cA}{\mbox{${\mathcal A}$}}
\newcommand{\cB}{\mbox{${\mathcal B}$}}
\newcommand{\cV}{\mbox{${\mathcal V}$}}
\newcommand{\wcLs}{\mbox{${\widetilde {\mathcal L^s}}$}}

\newcommand{\wcF}{\mbox{${\widetilde {\mathcal F}}$}}
\newcommand{\cL}{\mbox{${\mathcal L}$}}
\newcommand{\cG}{\mbox{${\mathcal G}$}}
\newcommand{\wT}{\mbox{${\widetilde T}$}}
\newcommand{\wP}{\mbox{${\widetilde P}$}}
\newcommand{\wN}{\mbox{${\widetilde N}$}}
\newcommand{\cF}{\mbox{${\mathcal F}$}}
\newcommand{\cT}{\mbox{${\mathcal T}$}}
\newcommand{\cE}{\mbox{${\mathcal E}$}}
\newcommand{\cP}{\mbox{${\mathcal P}$}}
\newcommand{\mt}{\mbox{${\widetilde M}$}}
\newcommand{\rrrr}{\mbox{${\mathbb R}$}}
\newcommand{\zzzz}{\mbox{${\mathbb Z}$}}

\title[$\rrrr$-covered foliations and transverse pseudo-Anosov flows]{$\rrrr$-covered foliations and transverse pseudo-Anosov flows in atoroidal pieces}
\author[S.R. Fenley]{Sergio R.\ Fenley} 
\thanks{Research partially supported by Simons foundation
grants 2804429 and 637554.}
\address{Florida State University, Tallahassee, FL 32306, USA}
\email{fenley@math.fsu.edu}

\begin{document}
 
 \begin{abstract}
We study the transverse geometric behavior of $2$-dimensional
foliations in $3$-manifolds.
We show that an $\rrrr$-covered transversely orientable
foliation with Gromov hyperbolic leaves in a closed
$3$-manifold admits a regulating,
transverse pseudo-Anosov flow (in the appropriate sense)
in each atoroidal piece of the
manifold. The flow is a blow up of a one prong
pseudo-Anosov flow.
In addition we show that there is a regulating flow
for the whole foliation.
We also determine how deck transformations act on the 
universal circle of the foliation.


\end{abstract}

\maketitle


\section{Introduction}

This article studies the transverse geometric behavior of
$2$-dimensional foliations in $3$-manifolds.
We will assume that the foliation is transversely orientable
so there is a transverse flow. Any such flow can be used to
understand how the geometry of leaves varies transversely $-$
at least locally.
By change in geometry in this article
we mean the following: consider a 
geodesic arc in a leaf of the foliation and use the
chosen transverse flow to push this arc to an arc
in a nearby leaf. Does the length increase or decrease
and by how much? Notice that there are other important
ways to consider changes in transverse geometry:
for example consider
how the spacing between distinct leaves varies, which leads
to holonomy invariant transverse measures.

The obvious first example to analyze is when the foliation
is a fibration, where we consider the case of closed
$3$-manifolds. The foliation is encoded by the monodromy,
which is a homeomorphism of a closed surface. By the
Nielsen-Thurston theory \cite{Th4,Bl-Ca}, the 
monodromy is either homotopically
periodic, reducible or pseudo-Anosov.
Reducible means that there is a simple closed curve and
a power preserves this curve.
Pseudo-Anosov means that the homeomorphism preserves 
a pair of singular, transverse one dimensional foliations,
whose leaves are either contracted by the map (stable)
or expanded (unstable). The singularities are $p$-prong
type with $p \geq 3$. The pseudo-Anosov option
very strongly describes how the geometry of
leaves is changing by some appropriate transverse suspension flow, 
describing the 
directions of contraction and expansion. 
This geometric information was crucial to geometrize such manifolds
\cite{Th1,Th2}.

In this article we study more general foliations.
One general problem is the use of a transverse flow.
In the case of fibrations any transverse flow induces
homemorphisms between leaves, so we can compare how it
affects the geometry. 
Whenever there is non trivial
holonomy of closed curves \cite{Ca-Co}, the transverse
flow cannot even take a closed curve to a closed curve.
This leads to the first adjustment: the transverse change of
geometry is best understood in the universal cover and
for Reebless foliations: then leaves in the
universal cover are simply connected \cite{No}, so any
compact set in a leaf can be pushed by the transverse flow to
nearby leaves. But in general this cannot be
accomplished for entire leaves.
In fact one necessary condition for the transverse flow 
in the universal cover to be a homeomorphism between arbitrary
leaves is
that the foliation is what is called $\rrrr$-covered
\cite{Fe1}: the leaf space of the foliation in the universal
cover is homeomorphic to the reals $\rrrr$.
In this article we will prove results about
$\rrrr$-covered foliations.

Fibrations are $\rrrr$-covered. But even in this situation the 
pseudo-Anosov case is best understood by looking
at the action on the ideal boundary as follows. Restrict to the case that
the fiber is negatively curved which is the generic
case, so one can assume the fiber is a
hyperbolic surface. The universal cover is the hyperbolic 
plane, compactified with an ideal circle \cite{Be-Pe}.
Any lift of the monodromy to the universal cover induces a homeomorphism
of this ideal circle. Thurston \cite{Th4,Bl-Ca}, following ideas of Nielsen,
did a very thorough study of this action on a circle, yielding
(in the non periodic, irreducible case)
invariant geodesic laminations on the surface, which blow down
to the singular foliations associated with the pseudo-Anosov
monodromy.

This program of seeing all ideal circles of leaves
in the universal circle as one single object
can be carried out to a certain extent
for any foliation with hyperbolic leaves: this is the theory
of the universal circle of foliations \cite{Th5,Th6,Ca-Du,Cal2},
which introduces a powerful way to collate all circles at
infinity into a single circle,
called the univeral circle of the foliation. This has
some powerful consequences for the geometry of the foliation
and the manifold \cite{Ca-Du,Cal2}.
The general expectation is that either geometry does not change 
very much transversally $-$ yielding a Seifert fibered space
structure; or there is some region with unbounded distortion yielding
to some pseudo-Anosov behavior in at least part of the manifold.

This strategy has been carried out very successfully when the
foliation is $\rrrr$-covered (again the case of Gromov hyperbolic
leaves is the generic case) \cite{Fe1,Cal1}, and $M$ is
atoroidal. The atoroidal case is the most common one
as the manifold is then atoroidal and hence
hyperbolic by Perelman's results.
In this case it was proved in \cite{Fe1,Cal1}
that (when the foliation is transversely orientable),
there is a pseudo-Anosov flow transverse to the foliation and
regulating. {\em Regulating} means that in the universal cover every
lifted flow line intersects every leaf of the foliation.
One of the important consequences is that this provided a proof
of the weak hyperbolization conjecture: either there is a $\mathbb Z^2$
subgroup or the fundamental group of the manifold is Gromov
hyperbolic \cite{Gr}. This result on Gromov hyperbolicity 
was of course superseded by the full
proof of the geometrization conjecture by Perelman.

What was left unanswered in \cite{Fe1,Cal1} is the question of what
happens in the intermediate case: that is when $M$ is not Seifert fibered
or atoroidal. In particular the JSJ decomposition of $M$ is not trivial.
The purpose of this article is to analyze the transverse geometry
of $\rrrr$-covered foliations in this intermediate situation.
Our main result is the following:

\begin{theorem} (Main theorem) \label{thm.main}
Let $\cF$ be a two dimensional foliation in $M^3$
closed so that $\cF$ is transversely oriented, $\rrrr$-covered, and
has Gromov hyperbolic leaves. Let $P$ be an atoroidal piece of the
JSJ decomposition of the manifold. Then there is a flow $\Phi$ in $P$
which is a blow up of a one prong pseudo-Anosov flow, so that
$\Phi$ is transverse to $\cF$ restricted to $P$ and it is regulating
for $\cF$ restricted to $P$.
The union of the regular periodic orbits of $\Phi$ is
dense in $P$.
\end{theorem}

Once this result is proved it is not very hard to obtain the 
following consequences. The first involves the action of
$\pi_1(M)$ in the universal circle.

\begin{corollary}
Let $\cF$ as in the Main theorem and $\gamma$ in $\pi_1(M)$ 
associated with
a periodic orbit of $\Phi$ in the interior of $P$.
Then up to a finite iterate the action of $\gamma$ on
the universal circle of $\cF$ has finitely many fixed
points which are alternatively attracting and repelling.
If the orbit is regular there are exactly $4$ fixed points
up to iterate.
\end{corollary}

In section \ref{sec.background} we provide the necessary
background on the universal circle of foliations.

We also prove the following, which extends well known
results for $\rrrr$-covered foliations in atoroidal
manifolds \cite{Fe1,Cal1}:

\begin{corollary} Suppose that $\cF$ is a transversely oriented
foliation which is $\rrrr$-covered. Then there is a flow transverse
to $\cF$ which is regulating for $\cF$.
\end{corollary}

The Main Theorem also helps to understand in general
the action of $\pi_1(M)$
on the universal circle of the foliation $\cF$.

\vskip .1in
Besides their intrinsic interest, the results of this article, 
particularly the Main Theorem has some uses in other situations.
For example in partially hyperbolic dynamics in dimension $3$
the properties of two dimensional foliations are essential
due to the results of Burago-Ivanov \cite{Bu-Iv} who produced
some ``branching foliations" associated with the dynamics.
In many cases these foliations are $\rrrr$-covered and the
change in transverse geometry can give important information.
In particular in \cite{FP2} we use the results of this article
on transverse pseudo-Anosov
flows on atoroidal pieces and group actions on the universal
cover to prove that some partially hyperbolic
diffeomorphisms in dimension $3$ are not dynamically coherent,
amongst other results in \cite{FP2}.

\subsection{Ideas of the proof of the main theorem}

Let $\wcF$ be the foliation lifted to the universal cover $\mt$.
As indicated above the main idea is to use the universal circle
of the foliation. In the case of an $\rrrr$-covered foliation the
universal circle is canonically homeomorphic to the circle 
at infinity of any leaf of $\wcF$.
This is described more carefully in the next section.
To understand the change of geometry one uses geometric shapes
in the leaves of $\wcF$ determined by ideal points in the
leaves. Any pair of geodesics in the hyperbolic plane are isometric and
so are any ideal triangles. To see distortion one has to look
at ideal quadrilaterals in leaves of $\wcF$ determined by
four ideal points in these leaves. Now change the leaves and
move the ideal points according to the universal circle
identifications.
The ideal quadrilaterals change and may become thinner in one
direction or in the opposite direction. In \cite{Fe1} we used this
distortion quadrilaterals to produce a pair two dimensional immersed
laminations transverse to $\cF$ and intersecting leaves
of $\cF$ in geodesics.
These laminations capture some of the 
distortion in the transverse direction.
The atoroidal property was then used heavily to show that
these immersed laminations fill $M$ and they are in fact embedded,
and eventually produce the transverse pseudo-Anosov flow.

We continue this analysis in the toroidal case. It is expected
that the immersed laminations do not fill $M$. The
constructions and proofs in \cite{Fe1} are to a certain
extent specific to the atoroidal case. In particular
Theorem 5.1 of \cite{Fe1} produces what is called a leafwise
geodesic embedded lamination. The passage from an immersed
to an embedded lamination is fundamental in the understanding
of the problem. The process is done by a convex hull procedure.
In the toroidal case the geodesic lamination
obtained by this convex hull process  theoreticallly
can well be a cutting torus in the JSJ decomposition
manifold. We definitely do not want such a lamination (a torus),
as it would not
describe the transversal change of geometry. At this point the
proofs in the atoroidal and non atoroidal case diverge.
In the toroidal case we study in more detail the
actual immersed laminations which are the limits of the distortion
quadrilaterals,  
 we show that components of 
these laminations are contained in
the interior of the atoroidal pieces (appropriately
adjusted). In particular they cannot be any of the tori
of the JSJ decomposition. This is the hardest fact to prove and it
uses a lot the definition of the distortion quadrilaterals.
We first show that the geodesics produced by the limiting process
cannot cross the tori of the JSJ decomposition. Then we
show that the geodesics cannot be contained in the tori either.
The second property is harder to obtain and involves manipulating
the foliation we start with.
After this is done,
We will show that the a priori immersed
laminations are in fact embedded.
We will also show that the two laminations  we obtain are not the
same, they are transverse to each other and they fill an atoroidal
piece $P$. Then there is a blow down process to produce a 
pseudo-Anosov like flow $-$ but there may be one prongs. So we
generalize the notion of a pseudo-Anosov flow to a one prong
pseudo-Anosov flow. Finally, an appropriate blow up of this
flow produces a flow in $P$ which is transverse to $\cF$ in $P$
and regulating. 

Once this is done the construction of the regulating flow for the
whole foliation is not so complicated.

\vskip .1in
\noindent
{\em Acknowledgements:} We thank Rafael Potrie for many
comments and suggestions that helped improve this article.

\section{$\rrrr$-covered foliations with Gromov hyperbolic leaves}
\label{sec.background}

Here we explain the basics about these foliations
and review the information we need about them.
The details are in \cite{Fe1}.
Let $\cF$ be an $\rrrr$-covered foliation
with Gromov hyperbolic leaves. By Candel's theorem
\cite{Cand} there is a metric in $M$ making every leaf of $\cF$
into a hyperbolic surface (notice that
$\rrrr$-covered is not necessary
for Candel's theorem).
Let $\wcF$ be the lifted foliation to $\mt$. Each leaf $F$ of $\wcF$
with its induced Riemannian metric is isometrict to the
hyperbolic plane and is compactified with an ideal circle $S^1(F)$.

First we introduce the ideal annulus $\cA$.
As a set $\cA$ is the union of $S^1(E)$ 
where $E$ are the leaves of $\wcF$.
The topology is as follows: consider a transversal $\tau$ to
$\wcF$. For each point $x$ in $\tau$ with $x$ in $E$ leaf
of $\wcF$, consider the unit tangent
bundle of $E$ at $x$ which is a circle. Each
unit vector determines a geodesic ray in $E$ starting
at $x$ with that direction. This determines an ideal point in $E$, hence
a point in $S^1(E)$. The map between directions
and $S^1(E)$ is a bijection for each $E$.
As $x$ varies in $\tau$ this provides a bijection between the
unit tangent bundle of $\wcF$ restricted to $\tau$ and
the union of the circles at infinity of the leaves intersecting
$\tau$. The union of the unit tangent bundles to $\wcF$ along
$\tau$ has a natural topology coming from the geometry 
of $\mt$. We put a topology in $\cA$ induced by these local bijections.
In \cite{Fe1} we proved that this topology is well defined, and
deck transformations act by homeomorphisms on $\cA$.
Notice that this topology in $\cA$ clearly induces the
natural topology in each $S^1(E)$.
One important property is the following: suppose that 
$\alpha$, $\beta$ are continuous curves in $\cA$ transverse to the foliation 
by circles at infinity of leaves. Suppose that for each 
$E$ in $\wcF$ the intersection of $\alpha, \beta$ with $S^1(E)$,
denoted respectively by $\alpha_E, \beta_E$, are distinct points
in $S^1(E)$. As $E$ varies let $\ell_E$ be the geodesic in $E$
with ideal points $\alpha_E, \beta_E$. Then the geodesics
$\ell_E$ vary continuously in $\mt$ with $E$.

We now describe the universal circle $\cU$ of $\cF$. 
There are two possibilities:

\begin{itemize}
\item
The {\em uniform case} $-$ Here for any two leaves $E, F$ of $\wcF$,
the Hausdorff distance between them (as subsets of $\mt$) is finite
\cite{Th5}.
The bound obviously depends on the pair of leaves.
For any pair of leaves $E, F$ there is a map 
$\tau: E \rightarrow F$ which is a quasi-isometry. This quasi-isometry
is coarsely well defined and induces a homeomorphisms $\tau_{\infty}$
between
$S^1(E)$ and $S^1(F)$. The map $\tau_{\infty}$ is as follows: given
$p$ in $S^1(E)$ there is a unique $q$ in $S^1(F)$ so that if $r$
is a geodesic ray in $E$ with ideal point $p$ and $r'$ is a geodesic
ray in $F$ with ideal point $q$ then $p, p'$ are a finite Hausdorff
distance from each other in $\mt$.
The homeomorphisms between ideal circles satisfy a cocycle property
and they are obviously equivariant under the action of deck
transformations.

\item
The {\em non uniform case} $-$ In particular there are no compact leaves.
In this case $\cF$ has a unique minimal sublamination $\cF'$. The 
complementary regions of $\cF'$ are $I$-bundles over non compact
surfaces. One can then collapse the complementary regions to produce
a new $\rrrr$-covered foliation which is minimal. All the results
proved for this induced minimal foliation pull back to $\cF$.
So when necessary we assume in this case that $\cF$ is minimal.
Under this condition and not uniform then for any leaves $E, F$
of $\wcF$
there is a dense set of directions in $E$ (and in $F$ too) so that
if $r$ is a ray in $E$ with one of these directions, then $r$ is
asymptotic to $F$. In fact it is asymptotic to a geodesic ray in $F$.
This gives a way to identify a dense set of points in $E$ with
a dense set of points in $F$. This extends to a unique homeomorphism
$\tau_{\infty}$
between $S^1(E)$ and $S^1(F)$. Again these homeomorphisms satisfy
a cocycle and are equivariant under deck transformations.
\end{itemize}

The universal circle $\cU$ of $\cF$ is the quotient of $\cA$
by these identifications: that is, $x$ in $S^1(E)$ is identified
with $y$ in $S^1(F)$ if $\tau_{\infty}(x) = y$, where $\tau_{\infty}$
is the map described above.

In both the uniform and non-uniform 
cases the universal circle induces a vertical foliation in 
$\cA$: two points in $\cA$ are in the same leaf of the vertical
foliation if they represent the same point of the universal circle.
This foliation is by continuous curves in $\cA$.

\vskip .1in
We now introduce the ideal quadrilaterals and parallepipeds in
$\mt$. Let $a, b, c, d$ be $4$ distinct points in $\cU$,
which are circularly ordered. Let $J$ be a compact interval
in the leaf space of $\wcF$. For each $F$ leaf of $\wcF$ in
$J$ let $Q_F$ be the ideal quadrilateral in $F$ with
ideal points $a_F, b_F, c_F, d_F$ which are the representatives
of $a, b, c, d$ in $S^1(F)$. 
By the properties of the universal circle, the quadrilaterals
$Q_F$ vary continuously with $F$.
Let $\cP$ be the union of $Q_F$ over all $F$ in $J$. 

In \cite{Fe1} we proved the following: if $M$ is not Seifert
then there is a sequence of parallelepipeds $\cP_i$ with
tops $Z_i$ and bottoms $X_i$ so that $Z_i$ is very thin in
one direction and $X_i$ is very thin in the opposite direction.
The thinness of a quadrilateral is the minimum distance between
opposites sides of the quadrilateral. An ideal quadrilateral is
regular if both such minimum distances are equal.
See Figure \ref{figure1} for a depiction of such a parallelepiped.

\begin{figure}[ht]
\begin{center}
\includegraphics[scale=1.20]{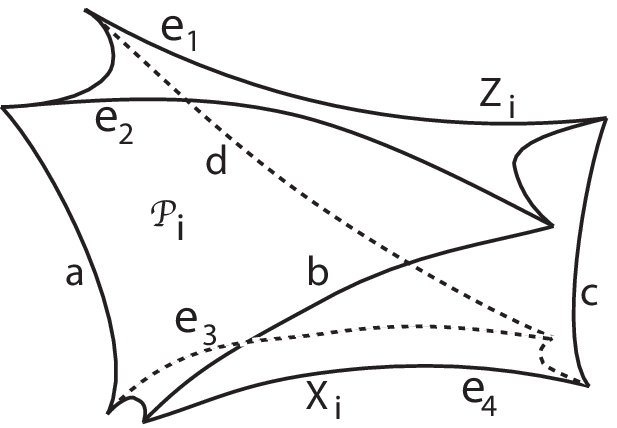}
\begin{picture}(0,0)
%
%
\end{picture}
\end{center}
\vspace{0.0cm}
\caption{A parallelepiped $\cP_i$: this is a $3$-dimensional
set in $\mt$ made up of ideal quadrilaterals in an interval of 
leaves of $\wcF$. The top quadrilateral is $Z_i$ and the
bottom one is $X_i$. The curves $a, b, c, d$ in the figure
are not made up of points in $\mt$, rather they are curves
of ideal points of leaves of $\wcF$. Each of these curves
denotes ideal points of different leaves, but corresponding
to the same point in the universal circle. On the top
the ideal quadrilateral $Z_i$ is thin in one direction: the
geodesic sides $e_1, e_2$ are very close to each other
in the respective leaf of $\wcF$. In the bottom, the 
ideal quadrilateral $X_i$ is thin in the opposite direction:
the geodesics $e_3, e_4$ are now close to each other in
the respective leaf of $\wcF$.}
\label{figure1}
\end{figure}

In \cite{Fe1} we construct these so that thinness of $Z_i$ converges
to $0$ in one direction while thinness of $X_i$ converges to
$0$ in the other direction.
These parallelepipeds are measuring the distortion of the
geometry transversally to $\wcF$. Pick a height $F$ where
$Q_F$ is a regular quadrilateral. Going up to the top $Z_i$ 
makes the quadrilateral very thin $-$ in other words stretching
the leaf in the direction of the side of the quadrilateral
which are very near and contracting the other  direction.
Going down the opposite happens.

Since the thinness of $Z_i$ converges to $0$ they are getting closer
and closer to geodesics. Project to $M$ and consider such a limit
geodesic $\ell_0$. Lift to a geodesic $\ell$ in a leaf $F$ of $\wcF$.
Now saturate $\ell$ by the universal circle, that is take all geodesics
in $E$ leaf of $\wcF$ so that ideal points correspond to the same
points in the universal circle $\cU$ as the ideal points of $\ell$.
The union of all of these is a closed set in $\mt$. The projection
is the a priori immersed lamination $\cL^s$, with lift $\wcLs$.
Each leaf $L$ of $\cL^s$ intersects the leaves of $\wcF$ in geodesics.
We call $\cL^s$ a {\em leafwise geodesic lamination}.
In the same way considering limits of the 
bottoms of the parallelepipeds produces the
immersed lamination $\cL^u$.

We state these results formally for future reference:
Given a geodesic $\ell$ in a leaf $F$ of $\wcF$ the 
{\em saturation} of $\ell$ is the union over all $E$ leaves
of $\wcF$ of the geodesic $\ell_E$ in $E$, so that the
ideal points of $\ell_E$ in $S^1(E)$ and the ideal points
of $\ell$ in $S^1(F)$ are the same points
under the universal circle identification.
We also call this the saturation of $\ell$ by the universal
circle. In \cite{Fe1} it is proved that this is a properly
embedded plane in $\mt$. We call it a {\em wall}.

\begin{proposition} \label{prop.laminations}
Let $\cP_i$ be a sequence of parallelepipeds in $\mt$ with tops
$Z_i$
very thin in one direction and bottoms $X_i$ very thin in the other
direction. Let $\cL^s$ be the leafwise geodesic lamination
obtained as follows: consider all limits of deck translates
of $Z_i$. These form a collection of geodesics
 in leaves of $\wcF$. Saturate these geodesics by the
universal circle.  
$\cL^s$ is the projection
of this collection of saturated walls
to $M$. Let $\cL^u$ be the immersed geodesic lamination
obtained by doing the same procedure with the bottoms $X_i$.
\end{proposition}

Notice that since different walls may intersect, we keep track
of the leaves of $\cL^s$ and not just the set.

\subsection{JSJ decomposition and $\rrrr$-covered foliations}
\label{ss.jsj}

Here we review some results from from \cite{FP1}.
The JSJ decomposition of an irreducible manifold splits
it into Seifert and atoroidal pieces \cite{He,Ja,Ja-Sh}.
Since we are considering non orientable manifolds we allow
Klein bottles amongst the cutting surfaces. 
Let $M$ be a $3$-manifold with a Reebless, $\rrrr$-covered
foliation with hyperbolic leaves.

Let $T$ be a torus or Klein bottle in the JSJ decomposition
and $\widetilde T$ be a lift to $\mt$.
Then $\widetilde T$ with its path metric
is quasi-isometrically embedded in $\mt$. This follows from
\cite[Theorem 1.1]{Ka-Le}, see also \cite[Section 3.1]{Ng}.
Then one can isotope $T$ so that $\widetilde T$ intersects
each leaf $F$ of $\wcF$ in a single component which is a
quasigeodesic in $F$. Then one can pull tight these quasigeodesics,
and assume that $T$ satisfies that 
$\widetilde T$ intersects
leaves of $\wcF$ in geodesics. 
We always assume this is the case for any $T$ a cutting
surface of the JSJ decomposition.
We say that $T$ is in {\em good position}.

In addition we have the following very important fact
(Proposition 4.4 of \cite{FP1}):
for any $F$ leaf of $\wcF$ let $\ell_F = \widetilde T \cap F$,
a geodesic in $F$ with ideal points $a_F, b_F$ in $S^1(F)$.
Then the set of $b_F$ as $F$ varies in $\wcF$ is a leaf of
the vertical foliation in $\cA$. In other words all $b_F \in S^1(F)$
correspond to the same point in the universal circle $\cU$ of $\cF$.

\section{Properties of the immersed leafwise geodesic
laminations}

In this section we prove the main ingredients to produce
the pseudo-Anosov flow in an atoroidal piece $P$.

\begin{proposition} \label{prop.crossing}
 Let $\cL^s, \cL^u$ be the immersed
leafwise geodesic laminations transverse to $\cF$, as
in Proposition \ref{prop.laminations}. Then no leaf of $\cL^s$ 
tranversely
intersects a torus or Klein bottle of the JSJ decomposition.
\end{proposition}

\begin{proof}
Suppose that some $L$ in $\wcLs$ transversely
intersects a lift
$\widetilde T$ for $T$ one of the tori or Klein bottles
of the JSJ decomposition.
If necessary lift to a double cover and we can
assume that $M$ is 
orientable. Hence we can assume that $T$ is a torus.

Recall that both $\widetilde T$ and $L$ intersect leaves of $\wcF$
in geodesics so that endpoints are constant under the
universal circle identification.
Therefore this transverse intersection of $\widetilde T$ and $L$ 
is seen in every leaf of $\wcF$. In particular up to deck
transformations there are parallelpipeds $\cP_i$ so that
the top quadrilaterals, call them  $Z_i$,
 converge to $\ell = F \cap L$ for
some $F$ in $\wcF$, see construction of $\cL^s$ and
section \ref{sec.background}.
The quadrilaterals $Z_i$ intersect $L \cap F$ in segments whose
length converge to zero. 
We can assume that all the tops of the parallelepids are in $F$.
Then $Z_i \cap F = \mu_i$ are segments converging to 
$L \cap F \cap \widetilde T = p$.
Let $\eta = F \cap \widetilde T$.

The bottoms of the parallelepipeds $\cP_i$, call them 
$X_i$ are quadrilaterals that are very thin in the other direction.
They still intersect $\widetilde T$ in a geodesic arc. 
This is because the parallelepid intersects leaves of 
$\wcF$ in ideal quadrilaterals with ideal points constant
when identified with the universal circle.
Now this geodesic arc is not very short, as was the case for
the top quadrilaterals, but rather very long. Call these geodesic
arcs $\zeta_i$. These project in $M$ into $T$. Since the foliation induced
by $\cF$ on $T$ is minimal, by adjusting the bottom
slightly, we can assume that $\pi(\zeta_i)$ contains
$\pi(\mu_i)$ and distance of $\pi(p)$ from both endpoints of $\pi(\zeta_i)$
along $\pi(\zeta_i)$ goes to infinity with $i$.

Let $\gamma_i$ be deck transformations of $\mt$ which are
also deck transformations of $\widetilde T$ which take
$\zeta_i$ to segments in $\eta$ containing
$\mu_i$ and so that distance along $\eta$
from $p$ to endpoints of $\gamma_i(\zeta_i)$ goes to infinity.
This is the crucial property: we are using that $T$ is compact
so we can always bring (by deck transformations)
long segments of the lifted foliation
$\wcF \cap \widetilde T$ to intersect a compact set.

We now analyze the action of $\gamma_i$ on the universal circle
$\cU$.
We use the identification of $\cU$ with $S^1(F)$. 
Here $\widetilde T \cap F$ is the geodesic $\eta$ in $F$.
Let $I$ be a complementary interval in $S^1(F)$ of the ideal points
of $\eta$.
We can parametrize $I$ as follows:
for each $q$ in $I$ there is a unique $q'$ in $\eta$ 
so that the geodesic ray in $F$ from $q'$ with ideal point $q$
is perpendicular to $\eta$. In this way
$I$ is parametrized
by $\eta$. The action of $\pi_1(T)$ on $\cU$ preserves 
the points of $\cU$ corresponding to the ideal points
of $\eta$ in $F$. In other words, when expressing this action
in terms of $S^1(F)$, it follows that $\pi_1(T)$ preserves
$I$. We analyze the action of $\pi_1(T)$ on $I$.
Since $I$ is canonically identified with $\eta$, this induces
an action of $\pi_1(T)$ on $\eta$.
Let this action be denoted by $\rho$.

Let the endpoints of $\mu_i$ be $x_i, x'_i$ and the
enpoints of $\gamma_i(\zeta_i)$ be $y_i, y'_i$. By renaming
we assume that $y'_i, x'_i, x_i, y_i$ are always linearly
ordered in $\eta$.
Therefore we have the following property:

\vskip .08in
\noindent
{\bf {Property 1}} $-$
We have points $x_i$ in $\eta$ converging to $p$ which are
taken by $\rho(\gamma_i)$ to $y_i$ which escape in $\eta$.
\vskip .08in

We use the following result. This result almost surely
has more hypothesis than what is needed to get a global
fixed point of a $\zzzz^2$ action on $\rrrr$, but it
suffices for our needs:

\begin{lemma} \label{lem.fixedpoint}
Let $G \cong \zzzz^2$ acting 
on $\rrrr \cong \eta$ so that there are points $x_i$ in $\rrrr$
in a compact set of $\rrrr$, and let $g_i$ in $G$ with $g_i(x_i)
\rightarrow \infty$ and $g_i$ has a fixed point $< x_i$.
For each $i$ let $z_i = \lim_{n \rightarrow -\infty} g^n_i(x_i)$.
Suppose that $d(z_i,x_i)$ converges to $0$ as $i \rightarrow \infty$
and $z_i$ converges to $z_0$ in $\rrrr$.
Then $z_0$ is a global fixed point of $G$.
\end{lemma}

\begin{proof}{}
Since $G \cong \zzzz^2$, 
the action is by orientation 
preserving homeomorphisms. 

Let $\beta$ in $G$ and suppose that $\beta$ does not
fix $z_0$. Up to taking an inverse we assume that $z_0 < \beta(z_0)$.
Now $x_i, z_i$ converge to $z_0$ as $i \rightarrow \infty$.
Fix $i$ big enough so that $\beta(z_i) > x_i$, but $\beta(z_i) < 
g_i(x_i)$. 
Then

$$g_i \beta(z_i) \ = \ \beta g_i(z_i) \ = \ \beta(z_i) \ > \ x_i.$$

\noindent
In other words $x_i < \beta(z_i) <  g_i(x_i)$,
and $\beta(z_i)$ fixed by $g_i$. This is a contradiction
and finishes the proof of the lemma.
\end{proof}

\noindent
{\em {Conclusion of the proof of Proposition \ref{prop.crossing}}}.
Let $g_i = \rho(\gamma_i)$  acting on $\rrrr \cong \eta$.
By Property 1, we have $x_i$ in $\rrrr$ with $g_i(x_i) = y_i$
and $y_i$ converging to $\infty$.
In addition $g_i$ has a fixed point $z_i$ in $[x'_i,x_i]$ and 
$[x'_i,x_i]$ converges to $p$. So $g_i, z_i, x_i$ satisfy
the other properties of Lemma \ref{lem.fixedpoint} with
$z_0 = p$.
The lemma shows that $\pi_1(T^2)$ has a global fixed
point. 

Consider the set $A$ of $\mt$ which in each leaf $E$ of
$\wcF$ is the geodesic ray in $E$ satisfying
\begin{itemize}
\item The starting point of $r$ is in $\widetilde T \cap E$
and $r$ is perpendicular in $E$ to $\widetilde T \cap E$.
\item The ideal point of $r$ in $S^1(E)$ and the global
fixed point of 
of the action of $\pi_1(T^2)$ on $I \subset S^1(F)$ 
correspond to the same point in the uninversal circle $\cU$.
\end{itemize}
Since $\widetilde T$ is a properly embedded plane, the
union of the geodesic rays $r$ as above forms an embedding
of a closed half plane in $\mt$.

We claim that $H = \pi_1(T^2)$ leaves $A$ invariant.
Any $\gamma \in \pi_1(T^2)$ leaves $\widetilde T$ invariant.
When acting on the universal circle $\cU$ then $\pi_1(T^2)$ 
fixes the point corresponding to the global fixed point
of $\pi_1(T^2)$ acting on $I$.
Let $E$ in $\wcF$ and $r = E \cap \widetilde T$. 
Then $\gamma(r)$ is contained in a leaf $D$ of $\wcF$, and 
since $\gamma$ is an isometry then $\gamma(r)$ is a geodesic
ray in $D$, $\gamma(r)$ starts in $\widetilde T$ and
$\gamma(r)$ is perpendicular to $\widetilde T \cap D$ in $D$.
Finally $\gamma(r)$ has ideal point in $S^1(D)$ which is
the same point as the global fixed point of $\pi_1(T^2)$ acting
on $I$ under the universal circle identification.
It follows that $\gamma(r) = D \cap \widetilde T$, hence
$\gamma$ preserves $A$.

In particular $\pi_1(T^2)$ leaves invariant
the infinite curve $\partial A$. 
This is impossible since $\pi_1(T^2)$ has to act
freely and properly discontinuously on $\partial A$.

This finishes the proof of 
Proposition \ref{prop.crossing}.
\end{proof}

We now prove a further property:

\begin{proposition} \label{prop.notori}
Suppose that $P$ is an atoroidal piece of $M$.
Then there is an immersed leafwise
geodesic   lamination $\cL^s$ (as in Proposition \ref{prop.laminations})
in $P$ and no leaf of $\cL^s$ isotopic to a component of
$\partial P$.
Similarly for $\cL^u$.
\end{proposition}

\begin{proof}
We do the proof for $\cL^s$.
In order to do this we will use the following doubling trick.
Cut $M$ along the components of the boundary of $P$ and
double $P$ along the boundary. Let this be the manifold $N$,
which we think of $P \cup P'$ where $P'$ is a copy of $P$.
Now do the whole analysis for $N$. 
As in the proof of Proposition \ref{prop.crossing}
we initially lift to a double cover if necessary and
assume that $N$ is orientable.

The manifold $N$
has an $\rrrr$-covered
foliation $\cF'$ which is the double of $\cF_{|P}$.
This folation $\cF'$ has hyperbolic leaves.
Let $\cV$ be the universal circle of $\cF'$.
If the action of $\pi_1(N)$ on $\cV$ is uniformly quasisymmetric,
then $N$ is Seifert fibered. In particular $P$ is Seifert 
fibered, contradiction to the hypothesis on $P$.
It follows that the action of $\pi_1(N)$ on $\cV$
is not uniformly quasisymmetric.
Hence there are parallelepipeds $\cP_i$ producing laminations in $N$
given by Proposition \ref{prop.laminations}, and 
still denoted by $\cL^s, \cL^u$. We will prove that they
induce laminations in $M$ as required.

By the previous proposition we know that no leaf 
$\cL^s$ intersects a boundary component of $P$ 
transversely.
If we prove that no boundary component of $P$ is a leaf of
$\cL^s$ then we get a sublamination of $\cL^s$ contained
in either $P$ or $P'$. If contained in $P'$ then the mirror
image is contained in $P$.

Suppose that a component $T$ of $\partial P$ is contained in $\cL^s$.
We use the setup of the proof of the previous proposition.
Suppose that some leaf of $\cL^s$ is a torus $T$
of the JSJ decomposition of $N$. Let $\widetilde T$ a lift.
As in the proof of the previous proposition there are parallelepipeds
denoted by $\cP'_i$ so that the tops are in a fixed leaf $F$ of $\wcF'$
and converge to $\eta = F \cap \widetilde T$.
Let the ideal points of $\eta$ in $S^1(F)$ be $x_1, x_2$.
We denote the tops by $Z'_i$. 
We will adjust the tops and produce a new set of parallelepipeds
$\cP_i$, still satisfying the thin conditions as before.

Let $a'_i,b'_i,c'_i,d'_i$ be the ideal points of $Z'_i$ in
$S^1(F)$, so that the geodesics $(a'_i,b'_i)$ and $(c'_i,d'_i)$
are very close in $F$ to $\eta$.
Up to renaming the points, assume that
$(a'_i,d'_i)$ are very close in $F \cup S^1(F)$ to $x_1$ and
$(b'_i,c'_i)$ are very close to $x_2$.
Now we do the symmetrization of $Z'_i$ with respect to $\eta$.
Recall that $N$ is the double of $P$, so there is a reflection
in the universal circle $\cV$ of $\cF'$ with respect to the
ideal points of $\eta$ (seen as points in $\cV$).
Denote this reflection map by $\xi: \cV \rightarrow \cV$.
Replace $Z'_i$ by an ideal quadrilateral $Z_i$ in $F$ with
ideal points $a_i, b_i, c_i, d_i$ as follows. Consider the 
pair $b'_i, c'_i$. Both are very close to $x_2$ and are distinct
from each other, so at least one is distinct from $x_2$.
If $\xi(b'_i) = c'_i$, we choose $b_i = b'_i$ and $c_i = c'_i$.
Otherwise one of $b'_i$ or $c'_i$ is farther from $x_2$ $-$ use
the reflection $\xi$ to compare them if on opposite sides of $x_2$. 
Suppose the farthest point is $b'_i$. Then let $b_i = b'_i$
and choose $c_i = \xi(b_i)$. Do the same for the pair $a'_i, d'_i$.
The resulting quadrilateral with ideal points $(a_i,b_i,c_i,d_i)$
is denoted by $X_i$. It has ideal points still very close to
$x_1, x_2$ respectively, so it is very thin in the same
direction that $X'_i$ is.
In particular by construction the thinness in this direction
goes to $0$ as $i \rightarrow \infty$.
But $X_i$ is less thin in this direction than $X'_i$ since we
may have pushed a pair of endpoints slightly farther away from
$x_1, x_2$ respectively.

Let now $\cP_i$ be parallelepiped intersecting the same set of leaves
of $\wcF$ that $\cP'_i$ intersects, but the top is now $Z_i$ instead
of $Z'_i$. Let $X_i$ denote the bottoms of $\cP_i$. 
The tops are very thin in the direction very close to $\eta$.
Since we made the tops slightly thinner in the opposite direction $-$
they are still very thick in that direction, but slightly less
thick, then the bottoms $X_i$ still have thinness converging
to $0$ in the opposite direction.
We explain a bit more: when moving the ideal quadrilaterals from
$X_i$ across leaves of $\wcF'$ using the universal circle to
move the ideal points, the following happens: 
the top quadrilateral $Z'_i$ moves to $X'_i$ which is very
thin in the opposite direction. Since $Z_i$ is slightly thinner than
$Z'_i$ in the opposite direction then $Z_i$ moves to even thinner
quadrilaterals $X_i$ in the opposite direction.

We use this sequence of parallelepipeds $\cP_i$.
Notice that by construction of $N$ as the double of $P$
the universal circle $\cV$ is symmetric with respect to $\eta$.
In other words this implies that for any leaf $E$ of $\wcF$
intersecting $\cP_i$ then the quadrilateral $Q^i_E = \cP_i \cap E$
in $E$ is symmetric with respect to $\widetilde T \cap E$.
In particular the bottom $X_i$ of $\cP_i$ is also symmetric 
with respect to the intersection of $\widetilde T$ with that
leaf.
In addition any deck translate of $\cP_i$ under an element
of $\pi_1(T)$ is still symmetric with respect to $\widetilde T$.

Orient the one dimensional foliation $\cF' \cap T$.

\begin{figure}[ht]
\begin{center}
\includegraphics[scale=1.00]{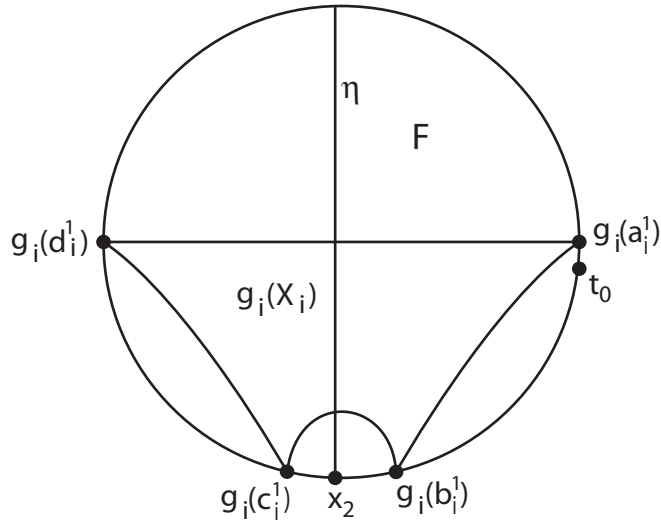}
\begin{picture}(0,0)
%
%
\end{picture}
\end{center}
\vspace{-0.5cm}
\caption{An example of action on the double $N = P \cup P'$.}
\label{figure2}
\end{figure}

We will use the setup of Proposition \ref{prop.crossing}.
Now $I$ is an open interval in the universal circle $\cV$
of $\cF'$.
Recall that $X_i$ are the bottoms of the parallelepipeds
$\cP_i$. Consider $X_i \cap \widetilde T$ which is a compact
segment denoted by $\zeta_i$ of the foliation induced by $\cF'$ in $T$.
Let $v_i$ be the positive endpoint of $\zeta_i$ with respect to
the orientation of $\cF' \cap T$.
Up to subsequence assume that 
$\pi(v_i)$ converges. Up to a slight modification of
the bottoms $X_i$ we can assume that $\pi(v_i)$ are
always in the same local leaf of the foliation in $T$.
So there are $g_i$ in $\pi_1(T)$ with $g_i(v_i)$ in a fixed leaf
$F$ of $\wcF$ which we can assume is $F$, and suppose that $g_i(v_i)$
converging to a point $v_0$. In other words
$g_i(v_i)$ are all in the curve $\eta = F \cap \widetilde T$. 
If the lengths of $g_i(\zeta_i)$ do not converge to infinity,
then the induced action of $g_i$ on $\eta \cong I$ 
has big intervals which are contracted to a bounded 
subcompact interval. This brings us to a setup that
Lemma \ref{lem.fixedpoint} produced a global fixed
point of the action of $\pi_1(T)$ on $I$. As seen
in the proof of Proposition \ref{prop.crossing} this
leads to a contradiction.

We conclude that the lengths of $g_i(\zeta_i)$ converge
to infinity. This means that if the other endpoint of $g_i(\zeta_i)$
is denoted by $q_i$ then $q_i$ escapes in $\eta$.
But the bottoms $X_i$ are symmetric with respect to $\widetilde T$
and the action of $\pi_1(T)$ on $\cV$ is symmetric with
respect to $\widetilde T$. 
Let the ideal points of $X_i$ be $a^1_i, b^1_i, c^1_i, d^1_i$
in $S^1(E_i)$, corresponding to $a_i, b_i, c_i, d_i$ respectively
in $S^1(F)$.
Here $E_i$ are the leaves of $\wcF'$ containing the bottoms
$X_i$ of the parallelepipeds $\cP_i$.
 By the above $g_i(a^1_i)$ converges to a point
$t_0$ in the interior of $I$. Also $g_i(b^1_i)$ converges
to an ideal point $x_2$ of $\eta$ in $S^1(F)$.
By the symmetry property, $g_i(c^1_i)$ converges to $t_1$
and $g_i(d^1_i)$ converges to the mirror image of $t_0$
on the other side of $\eta$. See figure \ref{figure2}.

This means that the sides $(g_i(d^1_i),g_i(a^1_i))$
and $(g_i(c^1_i),g^i(b^1_i))$ of $g^i(X_i)$ are not
getting close to each other in $F$.
Hence $(d^1_i,a^1_i)$ and $(c^1_i,b^1_i)$ are not getting
closer to each other in $E_i$ as well. This is a contradiction,
by construction of  the parallelepipeds $\cP_i$.

This proves that $T$ cannot be a leaf of $\cL^s$.
In fact we proved the following: for any choice of 
distortion parallelepipeds $\cP_i$ in $\wN$ so that
the tops $Z_i$ converge to a geodesic $\ell$ in any
leaf $F$ of $\wcF'$ then $\ell$ cannot cross any
lift of torus in $N$, nor be contained in any such lift.

\vskip .12in
\noindent
{\bf Dealing with a subtle point}

We obtained an immersed  leafwise geodesic
lamination $\cL^s$ in $N$ which is contained
in $P$ and it is obtained by taking limits of distortion
parallelepipeds $\cP_i$. In fact $\cL^s$ is contained in
the interior of $P$ and so induces an immersed leafwise
geodesic lamination in $M$. The subtle point is that
the distortion parallelepipeds $\cP_i$ are contained
in $\widetilde N$, but do not necessarily generate distortion
parallelepipeds in $\mt$ which will generate $\cL^s$ in $M$.
We will adjust our construction of the distortion 
parallelepipeds.

For notational reasons we will rename our parallelepipeds
$\cP'_i$.
We will adjust the $\cP'_i$ to obtain new parallelepipeds
(to be denoted $\cP_i$) with the properties we need.
What we want is that the distortion parallelepipeds
can be chosen contained in a fixed lift $\wP$ of $P$ to
$\widetilde N$. 

First we prove a preliminary fact. Consider the leaf $F$
of $\wcF'$ as in the beginning of the proof of the proposition.
The intersection $F \cap \wP$ is a hyperbolic surface with
geodesic boundary. The tops $Z'_i$ of the parallelepipeds $\cP'_i$ 
converge to a geodesic $\ell$ in $F$. We proved that $\ell$ 
is contained in the interior of $F \cap \wP$. 

We claim that $\ell$ is not asymptotic to a geodesic $g$ in $F$
which is the intersection of $\wT$ with $F$ for some 
JSJ surface $T$ of $N$ (notice that here we are not taking a double
cover of $N$ to make it orientable). 
Since we are taking all $\pi_1(N)$ translates and closures
of the limits
this means that some geodesic $E \cap \wT$ in $\wT$ ($E$
leaf of $\wcF'$) is contained in all the
limits. This we proved above that it  is impossible, proving the claim.
In addition any ideal point $p$ of $\ell$ in $F$ is accumulated
in $F \cup S^1(F)$ by geodesics $g_i$ which are intersections
of lifts of JSJ surfaces with $F$. Otherwise a half
plane in $F$ does not intersecting such a lift, and hence taking
limts a full leaf does not intersect such a lift contradiction.

Now we adjust the parallelepipeds $\cP'_i$. Let $\ell$ have
ideal points $x_1, x_2$. Let the tops of $\cP'_i$ be $Z'_i$ with
ideal points $a'_i, b'_i, c'_i, d'_i$ so that $a'_i, d'_i$ 
are very close to $x_1$ and $b'_i, c'_i$ very close to $x_2$.
We will enlarge the ideal quadrilateral $Z'_i$ still keeping
it very thin in the direction close to the geodesic $\ell$.
Let $I_i$ be the interval in $S^1(F)$ with ideal points
$a'_i, d'_i$ and very close to $x_1$. Recall that there are
sequences of 
endpoints of $\wT' \cap F$ for $\wT'$ lifts of JSJ surfaces
converging to $x_1$.  Hence for $i$ big we can choose
$a_i$ arbitrarily close to $a'_i$, $a_i$ not in the interior
of $I_i$ and $a_i$ an ideal point of $\wT' \cap F$ for 
some lift $\wT'$ of a JSJ surface. Do the same for
$d'_i, b'_i, c'_i$, producing $a_i, b_i, c_i, d_i$.
These are distinct and define an ideal quadrilateral $Z_i$.

Let $\cP_i$ the parallelepipeds intersecting the same leaves
of $\wcF'$ that $\cP'_i$ does but defined by the ideal
points $a_i, b_i, c_i, d_i$.
Now we prove properties of $\cP_i$. 
The geodesics $(a_i,d_i)$ are very close to $(a'_i,d'_i)$ 
and similarly $(b_i, c_i)$  are very close to $(b'_i, c'_i)$.
So $Z_i$ is very thin in the same direction that $Z'_i$ is.
Let $X_i$ be the bottoms of the parallelepipeds $\cP_i$.
The $Z_i$ are very thin in the $\ell$ direction, but 
less thin than $Z'_i$ in this direction.
This implies that the $Z_i$ are a little bit thinner than
the $Z'_i$ in the other direction.
This implies that the
$X_i$ are even thinner than the $X'_i$ in the opposite
direction.

We conclude that the $\cP_i$ have thinness in the top ($Z_i$)
in one direction converging to zero, and in the bottom ($X_i$)
thinness in the other direction converging to zero. Hence the $\cP_i$ are
distortion parallelepipeds which yield laminations $\cL^s$
and $\cL^u$.

We finally prove the crucial property we want.
By choice of the points $a_i, b_i, c_i$ and $d_i$ they are
ideal points of the hyperbolic surface
$F \cap \wP$. In particular the geodesics
$(a_i,b_i), (b_i,c_i), (c_i,d_i)$ and $(d_i,a_i)$ are contained
in $F \cap \wP$. Hence the tops $Z_i$ are entirely contained
in $\wP$. Since the quadrilaterals in $\cP_i$ are obtained
by following the universal circle and so is $\wP$,
it follows that 
$\cP_i \cap E$ is contained in $\wP$ for any leaf $E$ of $\wcF'$ that
it intersects. In particular $\cP_i$ is entirely contained
in $\wP$. This is the fact we wanted to prove.

Since $\cP_i$ is entirely contained in $\wP$ then we can
think of them also as contained in $\mt$. The quadrilaterals
in $\cP_i$ have the same leafwise geometry whether seen in $\wN$ or 
in $\mt$. It follows that $\cP_i$ are distortion quadrilaterals
in $\mt$, and of course they generate the immersed leafwise
geodesic laminations $\cL^s$ and $\cL^u$ in $M$.
At this point we can completely forget $N$ and consider all
objects $\cP_i, \cL^s, \cL^u$ in $\mt$ or in $M$.

This finishes the proof of the proposition.
\end{proof}

By the above proposition there is an immersed
 lamination $\cL^s$ contained
in the interior of $P$ and an immersed
  lamination $\cL^u$ contained in the interior of $P$.

\begin{remark}
This proof used the auxiliary manifold $N$
a doubling of $P$. A priori some other construction
using distortion parallelepipeds in $M$
could yield a lamination so that a component of $\partial P$
is a leaf. We strongly believe this is not possible, but
we are not able to prove this at this point.
\end{remark}

\begin{proposition}
The laminations $\cL^s, \cL^u$ constructed
in Proposition \ref{prop.notori} are embedded.
\end{proposition}

\begin{proof}{}
We will extensively use the analysis of \cite{Fe1}.
In pages 458 and 459 of \cite{Fe1} we showed that
there are 3 options for the immersed leafwise
geodesic laminations $\cL^s, \cL^u$:
\begin{itemize}
\item {\bf Option A.} \ No leaf of $\cL^s$ transversely
intersects another leaf of $\cL^s$. In other words $\cL^s$ is
an embedded leafwise geodesic lamination.
There is an analogous statement for $\cL^u$.
\item{\bf Option B.} \ No leaf of $\cL^s$ intersects 
transversely a leaf of $\cL^u$.
\item{\bf Option C.} \ There is a leaf of $\cL^s$ 
transversely intersecting a leaf of $\cL^u$.
\end{itemize}

\vskip .1in
In \cite{Fe1} pages 459-464 it is proved that Option C
implies Option A for both $\cL^s$ and $\cL^u$.
This is explicitly stated in Lemma 5.6 of \cite{Fe1}.

Fix a lift $\wP$ of $P$. 

Suppose that one of the laminations is not embedded,
without loss of generality suppose 
that $\cL^s$ self intersects transversely.
We will prove that option C holds in this case.
This will produce a contradiction since Option C implies
Option A for both $\cL^s$ and $\cL^u$.
This will prove that both $\cL^s$ and $\cL^u$ are embedded.

Suppose then that Option C does not hold.
By hypothesis there
are leaves $L, L'$ of $\wcLs$ contained in $\wP$
which intersect transversely.
We consider all leaves $L'$ of $\wcLs$ in $\wP$ so that
there is a sequence $L_0 = L, L_1, ...., L_k = L'$ of
leaves of $\wcLs \cap \wP$ with $L_i$ intersecting $L_{i-1}$
transversely for all $1 \leq i \leq k$.
This forms a subset $\cB$ of leaves of $\wcLs$ in $\wP$.

We consider the convex hull of the set of leaves in $\cB$.
envelope of the these set of leaves.
The boundary of the convex hull is made up of geodesics. When one
varies the leaf in $\wcF$ each boundary component of the 
convex hull varies according to the universal circle of $\cF$.
In addition there are no transverse self intersections
when projeting to $M$.
This is because it is the convex hull of these chains
of consecutily intersecting
leaves. 
In addition 
the projection of the boundary of the convex hull is made up
of compact surfaces, for the same reason,
see Thurston \cite{Th6} or \cite{Fe1}.
Hence these compact surfaces in $M$ are either tori or Klein bottles.
They are contained in $P$, hence have to be homotopic to 
boundary components of $P$.
But since these surfaces intersect leaves of $\cF$ in geodesics
it follows that these surfaces are all components of $\partial P$.
But no leaf of $\cL^u$ in $P$ can intersect any leaf
of $\cL^s$ in $\pi(\cB)$. It follows that some component of the
projection of the boundary of the convex hull separates 
$\pi(\cB)$ from some leaves of $\cL^u$ in $P$. This contradicts
that this projection is a component of $\partial P$. 

The contradiction implies that both $\cL^s$ and $\cL^u$ 
are embedded and finishes the proof of the proposition.
\end{proof}

Fix a piece $P$. Fix a lift $\widetilde P$ to $\mt$.
Let $\cL^s_m$ be a minimal sublamination of $\cL^s$,
and similarly define $\cL^u_m$ both leafwise geodesic
laminations in the interior of $P$.
We now know that they are embedded.
We will obtain properties of $\cL^s_m, \cL^u_m$.
In fact most of the properties 
are proved in \cite{Fe1}.
A {\em crown} is a hyperbolic surface which is a half
open annulus: its completion has one boundary component
which is a closed geodesic. There are
finitely many boundary components.
The  other boundary components are infinite geodesics,
which are consecutively asymptotic $-$ see \cite{Bl-Ca}.

\begin{lemma}\label{lem.properties1}
$\cL^s_m, \cL^u_m$ are not compact leaves. 
$\cL^s_m, \cL^u_m$ are distint and intersect transversely.
The complementary regions of
$\cL^s_m$ (or $\cL^u_m$) in $P$ are either
$S^1$ bundles over open finite sided ideal polygons (generating
open solid tori or solid Klein bottles) or $S^1$ bundles over
crowns, generating sets homeomorphic to torus $\times [0,1)$
or Klein bottle $\times [0,1)$.
\end{lemma}

\begin{proof}{}
We do the proof for $\cL^s_m$.
Suppose that $\cL^s_m$ is a compact leaf $B$. Since it has
a one dimensional foliation, then it is either
a torus or a Klein bottle. 
Since $P$ is atoroidal, then $B$ is isotopic to a component
of $\partial P$. 
Lift $B$ to a cover $\widetilde B$ boundedly isotopic to
a boundary component $S$ of $\widetilde P$. For each
$F$ leaf of $\wcF$ then $\widetilde B \cap F$ and $S \cap F$
are geodesics which are boundedly isotopic. Hence they are
the same geodesic. In other words $B$ is a component of
$\partial P$.
We proved in Proposition \ref{prop.notori}
that this is impossible.
This finishes the proof of the first assertion.

Now consider complementary regions. First of all $\cL^s_m$ is
not a foliation in $P$, since it does not intersect $\partial P$.
For each complementary region $V$ consider the boundary $S$
of the set in $V$ which is $\epsilon$ near $\cL^s_m$.
In \cite[Lemma 6.3]{Fe1} it is proved that each component
of $S$ is either a torus or a Klein bottle. 
If this
complementary region is not peripheral, then the torus
or Klein bottle is compressible and bounds a solid torus
or solid Klein bottle.
Proposition 6.1 of \cite{Fe1} further shows that $V$ is
a $S^1$ bundle over a finite sided ideal polygon.
If the region is peripheral then $S$ is isotopic to a component
$S'$ of the boundary. $S'$ is either a torus or Klein
bottle and $V$ is homeomorphic to $S' \times [0,1)$.
The other boundary components are in annular or M\"{o}bius band
leaves of $\cL^s_m$. As in Proposition 6.1 of \cite{Fe1} there
are finitely many of them, they are asymptotic, leading
to $V$ being an $S^1$ bundle over a crown surface.
This finishes the proof of the lemma.
\end{proof}


\begin{lemma} \label{lem.properties2}
$\cL^s_m, \cL^u_m$ are distint and intersect transversely.
The interior complementary regions of $A = \cL^s_m \cup \cL^u_m$ in $P$
are either finite sided polygons (with compact completion)  times $\rrrr$
or an $S^1$ bundle over a finite side polygon (with compact completion).
The peripheral complementary regions of $A$ have completion
torus or Klein bottle times $[0,1]$. The non peripheral boundary
is made up of compact annuli or M\"{o}bius bands contained
in leaves of $\cL^s_m$ or $\cL^u_m$.
\end{lemma}

\begin{proof}{}
The proof that they are distinct is exactly as in \cite{Fe1}.
Let $L'$ be a boundary leaf of (say) $\cL^s_m$. Then it is
an annulus or M\"{o}bius band with $\pi_1(L')$ generated by $\gamma$.
Let $L$ be the lift of $L'$ to $\mt$ with $\gamma(L) = L$.
This is analyzed in detail in Lemma 6.6 of \cite{Fe1}. There
the following is proved: 
Let $\theta(\gamma)$ be the action of $\gamma$
on the universal circle $\cU$ of $\cF$, and 
suppose that $\gamma$ is monotone increasing on the leaf space of $\wcF$.
Let $I$ be the open interval in $\cU$ determined by the
ideal points of $L \cap F$ for some fixed 
leaf $F$ of $\wcF$ $-$ under the identification
$S^1(F) \cong \cU$ satisfying the following:
there are leaves $E_n$
in $\widetilde \cL^s_m$ converging to $L$ and with ideal
points of $E_n \cap F$ in $I$ $-$ again under the identification
$S^1(F) \cong \cU$. Recall that $L$ is isolated on one side, but 
not on the other. This is because $\cL^s_m$ is minimal,
but not a compact leaf.
Then Proposition 6.6 of \cite{Fe1} proves that the action
of $\theta(\gamma)$ in $I$ is a contraction with a single fixed
point. In the case of the lamination $\cL^u_m$ the same
proposition shows that $\theta(\gamma)$ acts as an expansion in $I$.
These are contradictory and show that $\cL^s_m$ and $\cL^u_m$
are not the same lamination.

Since both $\cL^s_m, \cL^u_m$ are minimal it follows that
they do not share any leaf. By the properties of the complementary
regions of $\cL^s_m$ and $\cL^u_m$ (separately) it now follows
that they have to intersect transversely.
The components of the intersection
of complementary regions of $A$ with any leaf of $\cF$
have to have compact completion.
The description of the interior complementary regions of
$A$ is done in Proposition 6.11 of \cite{Fe1}.
The description of the peripheral components follows from the
description of the peripheral complementary components of $\cL^s_m$
and $\cL^u_m$ in $P$ separately, done in the previous Lemma.
\end{proof}

\section{One prong pseudo-Anosov flows and blow ups
in atoroidal pieces}
\label{sec.oneprong}

We generalize the notion of pseudo-Anosov flows to
include one prongs:

\begin{definition}
(one prong pseudo-Anosov flows) A flow $\varphi$ in a
closed $3$  manifold $Q^3$ is a one prong topological pseudo-Anosov 
flow if there are no point orbits of $\varphi$ and orbits
of $\varphi$ are contained in a pair of (possibly singular) two
dimensional foliation
$\cE^s, \cE^u$ weak stable and weak unstable of $\varphi$,
satisfing:
\begin{itemize}
\item All flow lines in a leaf of $\cE^s$ are forward 
asymptotic. In the backwards direction the orbits diverge
from each other in the intrinsic metric of the two dimensional
leaves. Similarly for $\cE^u$ with the reversed direction.
\item The (topological) singularities of $\cE^s, \cE^u$
are all of $p$-prong type where $p$ is a positive integer
which can be equal to one. The singular locus is a finite
union of periodic orbits of $\varphi$. The singular locus
of $\cE^s$ is the same as the singular locus of $\cE^u$.
\item The foliations $\cE^s, \cE^u$ are (topologically)
transverse to each other and intersect exactly along the
flow lines of $\varphi$.
\end{itemize}
\end{definition}

\begin{theorem} \label{thm.collapse}
Let $\cF$ be a transversely oriented $\rrrr$-covered
foliation with hyperbolic leaves in a $3$-manifold $M$. Suppose that
there is an atoroidal piece $P$ in the JSJ decomposition
of $M$. Then there is a one prong pseudo-Anosov flow
in a closed manifold $P_*$ obtained from $P$
by collapsing each boundary component of $P$ to a circle.
\end{theorem}

\begin{proof}{}
A part of this is 
done carefully in section 7 of \cite{Fe1}, which itself
just follows the constructions of Mosher \cite{Mos1,Mos2}.
There is a problem with the collapsing near each 
component of the boundary of $P$ which we will explain how
to adjust.
Let $A = \cL^s_m \cup \cL^u_m$.
For each leaf $F$ of $\cF$ restricted to $P$ we collapse
every closure of a component of 
$$F - (\cL^s_m \cup \cL^u_m)$$ 
\noindent
not intersecting the
boundary of $P$ to a point.
The laminations $\cL^s_m$ and $\cL^u_m$ collapse
to two dimensional foliations $\cE^s, \cE^u$
in the collapsed set.
Most of these closures of complementry 
regions in $F$ are compact quadrilaterals. Finitely
many of these closures which are also
in the interior of $P$ are finite
sided polygons with compact closure having
$2p$ sides. 
Here $p \geq 3$. The $p$ boundary leaves of $\cL^s_m$
associated with this complementary region collapse
to a $p$-prong singularity of $\cE^s$. Proposition 6.11
of \cite{Fe1} states that for every such complementary 
region of $\cL^s_m$ there is also a complementary
region of $\cL^u_m$ which intersects the leaf $F$ 
in a $p$-sided ideal polygon. So the same complementary
region of $A$ also generates a $p$-prong singularity of
$\cE^u$.

But there is a problem with the peripheral complementary components:
the problem in peripheral components is that the leaves 
of $\cF$ may not intersect them in sets with compact completion.
For the interior components, since $\cL^s_m, \cL^u_m$ fill $P$
then they are either solid tori or solid Klein bottles, and
a leaf $F$ intersects such a component locally in a finite
sided ideal polygon with compact closure in $P$. That is,
if you look at a local leaf $F$ intersecting the boundary,
and go around the boundary, the intersection with $F$ closes
up and bounds a disk in $F$ in the complementary component.
But for a peripheral complementary component $W$ 
look at how a leaf $F$ 
intersects the boundary of $W$ contained 
in $\cL^s_m \cup \cL^u_m$: when
you go around it may not close up. In fact if you continue
going around maybe it will be dense in this boundary 
component of $W$
and the collapsing of points in $F$ will be an awful 
topological space. 

In order to deal with this we do the following:
Let $\alpha$ be a closed curve which is an intersection
of a leaf $L$ of $\cL^s_m$ intersecting the boundary of $W$
with a leaf of $\cL^u_m$ intersecting the boundary of $W$.
Consider a local annulus $C$ in $L$ with one boundary
in $\alpha$ and entering $P - W$. Cut $P - W$ along this 
annulus $C$. Do this for all peripheral complementary 
components. 
Now do the collapsing along closures of intersections
of $F$ leaf of $\cF$ with the complement of $A = \cL^s_m \cup \cL^u_m$..

The laminations $\cL^s_m, \cL^u_m$ project to foliations
in the collapsed set. The curve $\alpha$ is still a closed
curve in the collapsed space. The intersection of the 
laminations $\cL^s_m$ and $\cL^u_m$ is a one dimensional
foliation in this object. It is orientable because the intersection
of $\cL^s_m$ and $\cL^u_m$ was
transverse to the foliation $\cF$ before collapsing.
Therefore the collapse of this intersection induces a flow.
We orient the flow
going in the negative direction transverse to the
foliation. Finally  glue the two sides of the opened
up annulus $C$ so that flow lines glue to flow lines.
The quotient is a closed manifold $P_*$ with two induced
foliations $\cE^s, \cE^u$. 

As proved in Proposition 7.2 of \cite{Fe1} orbits in the
same leaf of $\cE^s$ are forward asymptotic and in
a leaf of $\cE^u$ they are backwards asymptotic.
Hence this flow is a one prong pseudo-Anosov flow.
The reason for the possible one prongs is because of the
peripheral components of $P$: the boundary of a component
$W$ as above in the interior of $P$ may be made up of
a single annulus or M\"{o}bius band in $\cL^s_m$ and
a single annulus or M\"{o}bius band in $\cL^u_m$.
Then the collapsing will fold over the two sides of the
annulus or M\"{o}bius band leaves producing a one pronged
leaf for each of $\cE^s$ and $\cE^u$  in the quotient.
This situation is in fact extremely common: for example
consider the suspension case where $\cF$ is a fibration
over the circle and the monodromy has a pseudo-Anosov 
component. The pseudo-Anosov map associated with this monodromy
may have one prongs when collapsing a boundary component
to a point. This is exactly the same that happens here.

This finishes the proof of the theorem.
\end{proof}

\begin{figure}[ht]
\begin{center}
\includegraphics[scale=1.00]{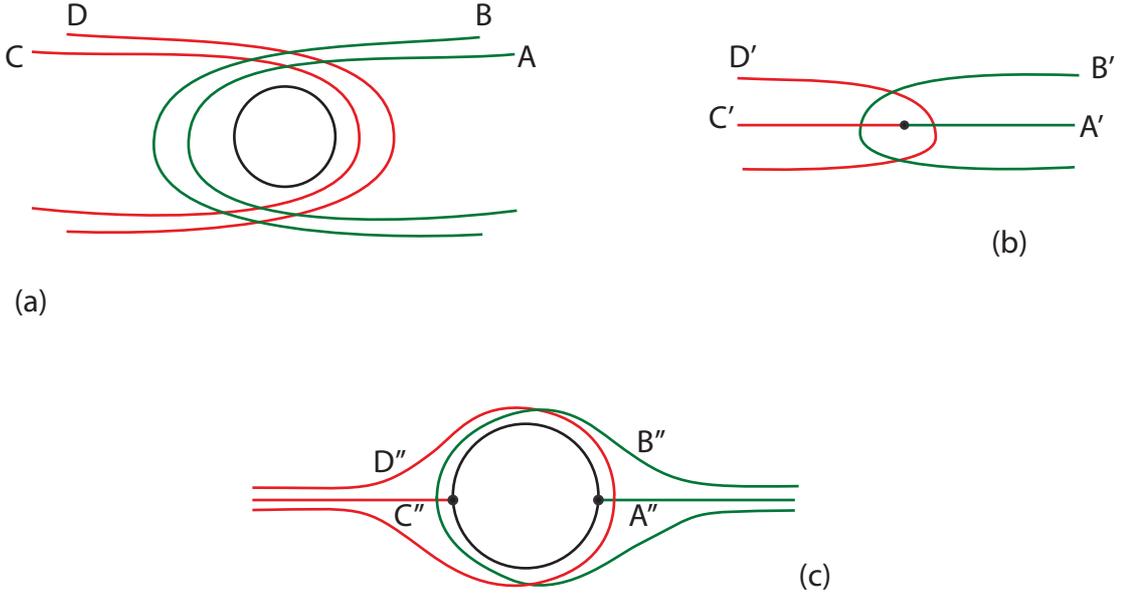}
\begin{picture}(0,0)
%
%
\end{picture}
\end{center}
\vspace{0.0cm}
\caption{Figures of the objects in the construction.
\ \ a) Depicts leaves of the laminations
$\cL^s_m, \cL^u_m$. The leaves $A, C$ are extremal,
there is a path from boundary of $P$ to these leaves not
intersecting any other leaf of the appropriate lamination.
\ \ b) The corresponding figure in $P_*$. The leaf $A$ of the lamination
collapses to a one prong leaf (this is the situation of a one prong), 
denoted
by $A'$ in this figure. The leaf $C$ collapses to $C'$. $A', C'$
intersect locally in the one prong depicted.
\ \ c) The blow up. The singularity blows up to a torus. $A'$ blows
up to the leaf $A"$ intersecting the boundary in a circle, 
similarly $C'$ blows to $C"$.}
\label{figure3}
\end{figure}

\subsection{Blow ups of one pronged pseudo-Anosov flows}

We now do blow ups of one pronged pseudo-Anosov flows
following Fried's method \cite{Fr}.
Consider a periodic orbit $\alpha$ of a one pronged pseudo-Anosov
flow in a manifold $Q$. The return map of flow lines
in a local a cross section satisfies the following:
in a prong of
the weak stable 
leaf of $\alpha$ is conjugated to $x \mapsto \frac{1}{2} x$.
In a weak unstable leaf it is conjugated to $x \mapsto 2 x$.
This obviously is only topological conjugation.
Fried \cite{Fr} blew up the orbit $\alpha$ to its projective
tangent bundle. The derivative of the return map
above induces a flow in the blown up set. This extends the
flow in $Q - \alpha$ to the blown up set with boundary.

Now return to the situation of Theorem \ref{thm.collapse}.
Given the construction of the laminations $\cL^s_m, \cL^u_m$
and the collapsed one prong pseudo-Anosov flow in $P^*$
the next result follows immediately.

\begin{corollary} \label{cor.regul1}
Under the hypothesis of Theorem
\ref{thm.collapse} the blow up of the one pronged pseudo-Anosov
flow in $P_*$ produces a flow $\Phi$ in $P$ which is
transverse to $\cF$ restricted to $P$ and which is regulating
for $\cF$ restricted to $P$.
\end{corollary}

As remarked above the singular foliations $\cE^s, \cE^u$ in $P_*$
may have one prong singularities. The blown up objects in $P$ are
denoted by $\cE^s_b, \cE^u_b$. They are foliations in the interior
of $P$ but are not foliations in the boundary. The are only points
in the boundary that are parts of leaves are those that come from
the blow up of the leaves in $P_*$ that originated from
collapsing the boundary of $P$. There only finitely many
of these in $\cE^s_b$ and in $\cE^u_b$. For example if
$\cE^s$ has a one prong in $\gamma$, then the blow up
of $\gamma$ will be a torus intersecting leaves of $\cE^s_b$
in a single component. This is illustrated in figure \ref{figure3}.
The figure shows the 3 steps in the process: Figure (a) depicts
the laminations $\cL^s_m, \cL^u_m$. These do not fill $P$, a local
cross section intersects each in a Cantor set. Figure (b) shows the
singular foliations $\cE^s, \cE^u$ in $P_*$. Figure (c) shows
the blow up of figure (b). The blow up of the foliations 
$\cE^s, \cE^u$ is a foliation in the interior of $P$, but not
on the boundary.

\subsection{Density of regular periodic orbits of $\Phi$ in $P$.}

We denote by $\Phi_*$ the one prong pseudo-Anosov flow in $P_*$.
The closed orbits of $\Phi_*$ which are obtained by collapsing
components of $\partial P$ are called the 
{\em boundary collapsed orbits}. These are exactly the orbits
that may be one prong orbits of the one prong pseudo-Anosov flow
$\Phi_*$.

\begin{proposition} \label{prop.dense}
The flow $\Phi_*$ is transitive, that is,
the set of periodic orbits of $\Phi_*$ is dense.
The same is true of the blow up flow $\Phi$ in $P$.
In fact, the set of regular periodic orbits of $\Phi$ in $P$
is dense in $P$.
\end{proposition}

\begin{proof}{}
If necessary lift to a double cover so that $P^*$
(and hence $P$ is orientable).

We will use a result about pseudo-Anosov flows. However the
flow $P_*$ may have one prongs and in general many results
do not hold for one prong pseudo-Anosov flows. 
To get around that we will do Fried's surgery. 
More specificlly do Fried's surgery 
on the boundary collapsed 
orbits to obtain $p$ prong orbits with $p > 1$.
Fried's surgery \cite{Fr} has two steps: blow up the
orbit using the action on the tangent bundle, then blow down
using a new meridian. 
The blow up is the procedure to go from
the one prong pseudo-Anosov flow $\Phi_*$ in $P_*$ 
to the blown up one prong
pseudo-Anosov flow $\Phi$ in $P$. 
The blow down is chosen by a new choice
of a meridian, call it $m$. If the intersection number of
$m$ with the collection of the blow ups of the
prongs of $\Phi_*$ is $p$, then
the resulting orbit is a $p$-prong orbit. One can always do
this so that $p \geq 2$, resulting in a pseudo-Anosov flow
$\Phi_2$ in a closed manifold $P_2$. The manifold $P_2$ is
obtained by Dehn surgery on $P_*$ on the curves associated with
the boundary collapsed orbits of $\Phi_*$.

What is important is that the flows $\Phi_*$ in $P_*$
and $\Phi_2$ in $P_2$ are what is called {\em almost equivalent}
\cite{De-Sh}:
$P_*$ minus the boundary collapsed orbits is the same as $P_2$ minus
the union  of the Dehn surgery orbits; and the flows restricted
to these sets have the same orbits.
Hence if the set of periodic orbits of $\Phi_2$ in $P_2$ is
dense in $P_2$ the same is true for the flow $\Phi_*$ in $P_*$.

We prove the result for $\Phi_2$ in $P_2$. Suppose that $\Phi_2$
is not transitive. Then Mosher \cite{Mos1} proved that there is
an incompressible torus $T$ transverse to the flow $\Phi_2$ and which
separates basic sets of $\Phi_2$. This transverse torus does not 
intersect periodic orbits of $\Phi_2$.
The blow up produces a torus $T'$ in $P$ transverse to the
blown up one prong pseudo-Anosov flow $\Phi$ in $P$. 
Since $P$ is atoroidal then $T'$ is isotopic to a boundary
component of $P$.
Projecting to $P_2$, it follows that  the torus
$T$ that bounds a solid torus
$B$ in $P_2$ containing a periodic orbit $\alpha$ associated with the 
blow down of the corresponding boundary component of $P$.
This orbit is obtained by a Dehn surgery of a possible one
prong orbit of $\Phi_*$ in $P_*$.

For simplicity assume the flow is outgoing from $B$ along
$T_2$.
Then the weak stable leaf of $\alpha$ cannot intersect $T_2$ $-$ since
$T^2$ is outgoing from $B$, and so this stable leaf 
is entirely contained in $B$. A lift of $B$ to the universal
cover $\widetilde P_2$ of $P_2$ is a solid tube with bounded cross section.
The lift $E$ of the weak stable leaf of $\alpha$ is a $p$-prong leaf
which is  properly
embedded in $\widetilde P_2$ \cite{Mos1,Ga-Oe}. This uses that
$\Phi_2$ is a pseudo-Anosov flow. This is not true in general
if there are one prongs. This is why we did the surgery
in $P_*$ to obtain $P_2$ and a pseudo-Anosov flow
$\Phi_2$ in $P_2$.
But $E$ is contained in a solid torus with compact cross sections,
so this is impossible.

We conclude that this is impossible. Hence the flow $\Phi_2$
is transitive and so is the flow $\Phi_*$ in $P_*$.
In particular $\Phi$ is also transitive in $P$. Since
there are finitely many possibly singular orbits of $\Phi$ in
$P$, then the union of the regular periodic orbits is
dense in $P$.

This finishes the proof of the proposition.
\end{proof}

\begin{remark}
After this proposition one may ask the following: there
is a lot of freedom in the initial collapsing map from $P$
to $P_*$ (collapsing the laminations $\cL^s_m, \cL^u_m$ 
to the singular foliations $\cE^s, \cE^u$ respectively).
Topologically the type of the singular 
foliations is determined by
the new meridian, which determines the topological type
of the collapsing. Hence why consider one prong pseudo-Anosov flows
in $P_*$ instead of always choosing a collapsing that yields
a true pseudo-Anosov flow? This is a valid question.
Here is one important reason to consider one prong pseudo-Anosov flows:
Suppose that in each component $D$ of $\partial P$ the foliation
$\cF$ induces a foliation by circles in $D$. 
For example this happens if the foliation in $\cF$ is a foliation
by compact surfaces, that is $\cF$ is a fibration over the
circle in $P$. Then the preferred collapsing is the one
that collapses each circle of $\cF_{| D}$ to a point. This
yields a foliation in $P_*$ which is transverse to the
induced flow $\Phi_*$, so $\Phi_*$ is a suspension flow.
But clearly the suspension flow $\Phi_*$ may have one prongs.
In this way the theory generalizes the theory of pseudo-Anosov
homeomorphisms of compact surfaces with boundary: it is well
known that pseudo-Anosov homeomorphisms with one prongs in
the boundary are extremely common and cannot be disregarded.
For example $\mathbb S^2 \times \mathbb S^1$ has a one
prong pseudo-Anosov flow which is a suspension of a one
prong pseudo-Anosov homeomorphism of $\mathbb S^2$. Obviously
$\mathbb S^2$ does not admit pseudo-Anosov homeomorphisms,
but admits one prong pseudo-Anosov homeomorphisms, for
example there is one with exactly 4 singularities, all
one prongs. The blow up of the suspension is a sphere
minus $4$ disks times $\mathbb S^1$.
There are also many other situations where the foliation 
restricted to each boundary component is by circles,
but the foliation in $P$ may not even have compact leaves.
Whenever this happens, the natural collapsing produces
a foliation in $P_*$ transverse to the flow $\Phi_*$.
\end{remark}

\subsection{Regulating flows transverse to $\rrrr$-covered foliations}

\begin{theorem} \label{thm.regul2}
Suppose that $\cF$ is a Reebless, $\rrrr$-covered,
transversely oriented foliation. Then $\cF$ admits a
transverse regulating flow. 
There is such a flow which is essentially flowing along
Seifert fibers in the Seifert pieces and is the
blow up one prong pseudo-Anosov flow in the atoroidal pieces.
\end{theorem}

\begin{proof}{}
We first deal with the case that $\cF$ admits
a holonomy 
invariant transverse measure.
See Definition.9.2.14 of \cite{Ca-Co} for a definition
of such a measure.
If $\cF$ has a compact leaf (hence clearly generating
a holonomy invariant transverse measure supported on it),
then we proved in \cite{Fe1} that
the compact leaf is a fiber. In addition
there is a suspension flow transverse to the foliation
which is regulating for $\cF$, see \cite{Fe1}.
Suppose that $\cF$ does not have compact leaves.
Then $\cF$ has
a unique minimal set and one can blow down complementary regions,
to produce a minimal foliation, see \cite{Fe1}.
Since the 
foliation is minimal the holonomy invariant measure has full support.
By a result of Tischler \cite{Tic} (see also
Theorem 9.4.2 of \cite{Ca-Co}),  it follows that $M$ fibers over
the circle. In addition the foliation can be approximated 
arbitrarily well by a fibration over the circle. Suspension flows
to the fibrations are also regulating for the foliation
$\cF$, see \cite{Th3}. This finishes the proof in this case.

From now on assume that there is no holonomy invariant transverse
measure. 
By Candel's theorem \cite{Cand} there is a metric in $M$
making all leaves hyperbolic, and we assume this metric.

If the JSJ decomposition of $M$ is trivial then $M$ is either
Seifert fibered or atoroidal. Consider the case that
$M$ is Seifert. By Brittenham's theorem \cite{Br} the foliation 
$\cF$ has a sublamination which is either vertical
or horizontal. If it is vertical then $\cF$ cannot be
$\rrrr$-covered. It follows that the sublamination is horizontal
and then so is $\cF$. Hence we can assume that
the Seifert
fibration is transverse to the foliation. Hence it is orientable,
so its lift to $\mt$ shows that the flow generated by
the Seifert fibration is regulating.
If $M$ is atoroidal the result is proved by the construction in 
\cite{Fe1}.

Assume from now on that the JSJ decomposition of $M$ is not trivial.
Using the results of section \ref{sec.background}
we assume that all decomposing tori and Klein bottles
of the JSJ decomposition are in good position with respect
with the foliation $\cF$.
Let $P$ be a piece of the JSJ decomposition.
If $P$ is Seifert the Seifert fibration provides
a transverse flow regulating for $\cF_{|P}$.
If $P$ is atoroidal Corollary \ref{cor.regul1} provides
a blow up of a one prong pseudo-Anosov flow in $P$ which
is transverse to $\cF_{|P}$ and regulating.

Now all one has to do is to match the flows in between
the pieces. Suppose that $T$ is a torus or Klein bottle
of the JSJ decomposition. 
On either side of $T$ there are pieces $P_1, P_2$ of
the JSJ decomposition with flows transverse to the
foliations $\cF_{|P_1}$ and $\cF_{|P_2}$ respectively.
Hence there are two flows in $T$ which are regulating
for $\cF_{T}$. 

We do an isotopy between  these flows in $T$ through flows transverse to
$\cF_{|T}$ and regulating for this foliation.
We do the case where $T$ is a torus, which is the 
more complicated case.
Choose a basis for the homology of $T$. All the foliations in $T$:
either $\cF_{|T}$ or the two foliations by flow lines do not
have Reeb components. Hence all the leaves in any of these foliations
have the same slope. The slope of $\cF_{|T}$ defines a half
line of slopes positively transverse to $\cF_{|T}$. Both 
the flows induced in $T$ are in this half space of positive slopes.
Then one can isotope one to the other keeping it in this half
space of positive slopes and keeping it regulating for $\cF_{|T}$.
Then enlarge $T$ to $T \times [0,1]$ and interpolate the flows
from one side to the other.

This constructs a flow transverse to $\cF$ and regulating for
$\cF$. This finishes the proof of the Theorem.
\end{proof}

\section{Action of the fundamental group on the universal 
circle}

Let $\cF$ be a transversely orientable, Reebless, $\rrrr$-covered
foliation in a closed $3$-manifold $M$ so that its leaves are
hyperbolic. We obtain some information about the action of 
deck transformations on the universal circle $\cU$ of $\cF$.

This builds up on several prior works: for example on actions
of lifts of homeomorphisms of closed surfaces on the
circle at infinity (see an account in the appendix of \cite{BFFP}).
The case of atoroidal manifolds, has already been worked out
previously, see \cite{Fe2}.

We first work out one specific example which is the direct
consequence of the results of this article.
First assume that $M$ has a non trivial JSJ decomposition.
Let $\varphi$ be a regulating flow as construted in Theorem 
\ref{thm.regul2}. We can assume that $\varphi$ preserves the
tori and Klein bottles of the JSJ decomposition. 
If there is
an atoroidal piece $P$ we may assume that $\varphi$ restricted
to $P$ is $\Phi$: a blow up of a one prong pseudo-Anosov flow.

Let $\widetilde \varphi$ the lift to $\mt$. 
Given any two leaves $E, F$ of $\wcF$ define 

$$\tau_{E,F}: \ E \  \rightarrow \ F, \ \ 
\tau_{E,F}(x) \ = \ \widetilde \varphi_{\mathbb R}(x) \cap F$$

\noindent
Since $\varphi$ is regulating for $\cF$ the maps $\tau_{E,F}$
are always homeomorphisms. They clearly satisfy a cocycle
condition.
Let $\wT$ be a lift to $\mt$ of a torus or Klein bottle in the
JSJ decomposition. Clearly $\tau_{E,F}(\wT \cap E) = \wT \cap F$.

If $M$ is Seifert or atoidal there is a transverse regulating
flow, so this defines a map $\tau_{E,F}$ for any $E, F$.

\begin{lemma} \label{lem.same}
For any $E, F$, the homeomorphism $\tau_{E,F}$ extends to
a homeomorphism, $\zeta_{E,F}: E \cup S^1(E) \rightarrow 
F \cup S^1(F)$. In addition for any $x$ in $S^1(E)$,
$x$ and $\zeta_{E,F}(x)$ define the same point in the
universal circle $\cU$ of $\cF$.
\end{lemma}

\begin{proof}{}
Consider first the case that $M$ is Seifert. Then we choose the
transverse flow $\Phi$ so that it is an isometry between 
the leaves in the universal cover, so it extends
as a homeomorphism between compactifications.
The foliation is uniform
and the map $\tau_{E,F}$ sends a point to a point boundedly
near it, so a geodesic ray to a geodesic ray boundedly near it.
Hence the ideal points in $E, F$ project to the same point in
the universal circle. If $M$ is atoroidal, then \cite{Fe1,Cal1}
proved that there is a transverse pseudo-Anoosov $\Phi$ to
$F$. This is constructed using the universal circle, with
the distortion quadrilaterals. In particular a leaf $L$ 
of $\cL^s_m, \cL^u_m$ satisfies that the ideal points of
$L \cap E$ as $E$ varies in $\wcF$ define the same point
in $\cU$. This implies the result in this case.

Finally suppose that $M$ has a non trivial JSJ decomposition,
so there is at least one torus or Klein bottle $T$ JSJ cutting 
surface. Here we need more information from \cite{FP1}.
Let $\wT$ be a lift of $T$ to $\mt$. Proposition 4.4 of \cite{FP1}
shows that $\wT \cap E$ as $E$ varies in $\wcF$ defines
a constant pair of points in $\cU$. 
Given $F$ in $\wcF$ let $\cG_F$ be the lamination in $F$ obtained
by intersecting all lifts $\wT$ of JSJ tori or Klein bottles $T$
with $F$. It is a lamination by geodesics.
Lemma 4.8 of \cite{FP1} states given $F$ in $\wcF$ then the
set of ideal points of $\cG_F$ is dense in $S^1(F)$
and for any non degenerate interval $J$ of $S^1(F)$ there are leaves
of $\cG_F$ with both ideal points in $J$.

Now fix $E, F$ leaves of $\wcF$. We know that for any $\wT$ lift
of a JSJ torus or Klein bottle then $\tau_{E,F}(\wT \cap E) =
\wT \cap F$. 
No two leaves of $\cG_E$ share an ideal point.
In addition the circular order in $S^1(E)$ induced by the
ideal points of leaves of $\cG_E$ is preserved by $\tau_{E,F}$:
the circular order induced in $S^1(F)$ by $\tau_{E,F}(\wT)$ 
as $\wT$ varies over the lifts is the same, when one identifies
$S^1(E)$ with $S^1(F)$ using the universal circle.

We know that for any ideal point $q$ in $S^1(E)$ 
either it is an ideal point of some leaf of $\cG_E$ or
is accumulated by ideal points of leaves of $\cG_E$ (so that
both endpoints converge to $q$).
These facts imply that $\tau_{E,F}$ extends to a homemorphism
from $E \cup S^1(E)$ to $F \cup S^1(F)$.
Since the homeomorphism satisfies that $q, \tau_{E,F}(q)$
project to the same point in $\cU$ for any $q$ ideal point
of leaf of $\cG_E$, it follows that this is true for all
$q$ in $S^1(E)$.
This proves the lemma.
\end{proof}

We fix a transverse, regulating flow $\varphi$ as above.
Let $\gamma$ in $\pi_1(M)$ be a deck transformation.
For any $E$ leaf of $\wcF$ define 

$$h_E \ = \ \gamma \circ \tau_{E,\gamma^{-1}(E)}$$

\noindent
This is a homeomorphism from $E$ to itself. By 
Lemma \ref{lem.same}
this induces a homeomorphism $\hi$ from $S^1(E)$ to itself.
Recall that $\tau_{E,\gamma^{-1}(E)}$ induces the identity
map in the universal circle level.
It follows that under the identification of $\cU$ with $S^1(E)$ 
then $\hi$ is the representation of the action of $\gamma$ 
on the universal circle $\cU$.

Suppose that there is at least one atoroidal piece $P$.
Recall the ``singular foliations" 
$\cE^s_b, \cE^u_b$ in $P$  (they are 
singular foliations in the interior of $P$).
Consider the lift of these to a lift $\wP$ of $P$ to $\mt$.
They induce foliations in $\wP$ so that intersected
with any leaf $E$ of $\wcF$ they are foliations by
quasigeodesics. Some are $p$-prong leaves, each prong is
a quasigeodesic. The transverse flow $\widetilde \varphi$ 
exponentially expands length along the unstable leaves 
and contracts along stable leaves.

Let $\gamma$ be an arbitrary element of $\pi_1(M)$ and
$\rho(\gamma)$ the induced action on the universal
circle $\cU$ of $\cF$.

\begin{proposition} \label{prop.superattracting}
Suppose that $M$ has an atoroidal piece $P$.
Let $\gamma$ be a deck transformation associated with an
interior periodic orbit of $P$. Then up to a finite power
$\rho(\gamma)$ has finitely many fixed points in $\cU$,
alternating between attracting and repelling.
In case $\gamma$ is associated with a regular orbit, then
$\rho(\gamma)$ (up to power) has exactly $4$ fixed points.

Finally if $E$ is in $\wcF$ and $q$ is an ideal point in
$S^1(E)$ associated with a fixed point of $\rho(\gamma)$,
the following happens: there is a neighborhood basis of 
$q$ in $S^1(E)$ defined by geodesics $\ell_i$ in $E$ so that
for any $x$ in  $\ell_i$ and 
$y$ in $\gamma \circ \tau_{E,\gamma^{-1}(E)}(\ell_i)$,
then $d_E(x,y) \rightarrow \infty$ if $i \rightarrow \infty$.
\end{proposition}

\begin{proof}{}
Fix a leaf $E$ of $\wcF$, we do the analysis in $E$. 
Let $\alpha$ be the periodic orbit associated with $\gamma$
and $\widetilde \alpha$ the lift of $\alpha$ fixed by
$\gamma$. Up to a power $\gamma$ fixes the stable and
unstable prongs of $\widetilde \alpha$. Let 
$z = \widetilde \alpha \cap E$. Let $\wP$
be the lift of $P$ containing $\widetilde \alpha$. 
Then the intersections of the stable and unstable prongs
of $\widetilde \alpha$ with $E$ are quasigeodesic rays
in $E$ entirely contained in $\wP$.
These prongs are contained in  leaves of 
$\widetilde \cE^s_b, \widetilde \cE^u_b$ respectively.

Fix one unstable prong $r$, that is, contained in 
$\cE^u_b(\widetilde \alpha) \cap E$. 
Fix a regular stable leaf $\zeta$
(that is a leaf of $\cE^s_b \cap E$)
intersecting $r$. Recall that $\hE  = \gamma \circ \tau_{E,\gamma^{-1}(E)}$.
Consider $\hiE(\zeta)$. The flow $\varphi$ preserves 
$\cE^s_b, \cE^u_b$, contracts length exponentially
along the stables and expands
along the unstables. Hence $\hE$ preserves the foliations
by stables and unstables in $\wP \cap E$. 
Then $\hiE(\zeta)$ converges to the stable leaf of $z$
when $i \rightarrow -\infty$  (if
$z$ is a $p$-prong it converges to properly embedded real line
in this leaf). Let this limit be $\zeta'$.
In addition $\hiE(r)$ escapes in $E$ when
$i \rightarrow \infty$. Notice that for 
any $i$, $\hiE(\zeta)$ is entirely contained in $\wP$.
In addition $\hiE(\zeta)$ are uniform quasigeodesics, independent
of $i$. As $i \rightarrow \infty$ they escape in $E$ and
are nested so they converge to a single ideal point, which
is the ideal point of $r$. 

Let $a_1, a_2$ be the ideal points in $S^1(E)$ of $\zeta'$ and
$b_1$ the ideal point of $r$. Let $I$ be the interval in $S^1(E)$
with endpoints $a_1, a_2$ and containing $b_1$.
The above shows that $h_{\infty}$ fixes $a_1, a_2, b_1$ and
acts as a contraction on the interior of $I$ with single
fixed point $b_1$. This proves the first assertion of the
proposition.

\vskip .1in
We now consider the last statement of the proposition.
Consider the geodesics which are obtained by pulling
$\hiE(\zeta)$ tight. These are leaves of $\widetilde \cL^s_m \cap E$.
These form a neighborhood basis of the ideal point $b_1$ of
$r$ in  $E \cup S^1(E)$.
These geodesics are also a uniform bounded distance in $E$ from 
$\hiE(\zeta)$ (for the same $i$) and so are their images
under $\hE$. 

The angle between leaves of $\cL^s_m \cap E$ and
leaves of $\cL^u_m \cap E$ is uniformly bounded below.
Let $\ell_i$ be the geodesic obtained by pulling
$\hiE(\zeta)$ tight. Let $r_g$ be the geodesic obtained
by pulling $r$ tight. The action of $\hE$ on $r$ expands
length exponentially. 
This follows because $\varphi$ expands length
exponentially. The geodesics $\ell_i$ are a bounded
distance from $\hiE(\zeta)$ and so is $r_g$ from $g$.
It follows that the distance from $\ell_i \cap r_g$
to $\ell_{i+1} \cap r_g$ converges to infinity as
$i \rightarrow \infty$. 
The angle condition implies that for any $x$ in $\ell_i$ and
$y$ in $\ell_{i+1}$ then $d_E(x,y) \rightarrow \infty$
as $i \rightarrow \infty$. Since the Hausdorff distance
between $\ell_{i+1}$ and $\hE(\ell_i)$ is uniformly
bounded, we obtain the bound desired.

If the ideal point is an ideal point of a stable prong,
then we use inverses instead in the above argument.

This finishes the proof of the proposition.
\end{proof}

\begin{remark} The last condition in the proposition is 
what is called {\em super attracting} in \cite{FP2}.
The super attracting definition is particularly useful
in the case that $\cF$ is uniform. Since we do not use
that here we do not define it formally, and we refer
the interested reader
to \cite{FP2}.
The specific result of this proposition is used in
\cite{FP2} to help analyze partially hyperbolic diffeomorphisms
homotopic to the identity in $3$-manifolds and to prove that
some of them are not dynamically coherent.
\end{remark}

\subsection{Some general properties of action on
the universal circle $\cU$.}

We are now ready to describe more general
information about the action
$\rho(\gamma)$ on the universal circle $\cU$. There are
too many cases to enumerate in the statement of a single
result. Instead we little by little describe 
each individual case.
The foliation $\cF$ satisfies the properties announced
in the beginning of this section.

\vskip .1in
\noindent
1) Suppose that $M$ is Seifert. As explained above we can assume
that the Seifert fibration is transverse to $\cF$, and we can put
a metric so that flowing along Seifert fibers is a local isometry
between leaves of $\cF$. Any deck transformation preserves
the Seifert fibration so induces an isometry on the quotient
of $\mt$ by the lift of the Seifert fibration. This quotient $R$
is isometric to the hyperbolic plane and the ideal circle of
this plane is canonically identified to the universal circle
$\cU$. Let $\gamma$ be a deck transformation, so it induces
an isometry of $R$.

$-$ If the isometry is elliptic then $\gamma$ is associated with
a fiber of the Seifert fibration. Then a finite power of $\gamma$ is
the identity on $R$. If $\gamma$ preserves orientation
then the action of $\rho(\gamma)$ on $\cU$ is either free or fixes every point.
If $\gamma$ reverses orientation on $R$ then $\rho(\gamma)$ has two fixed
points on $\cU$.

$-$ If the isometry induced on $R$ is hyperbolic there are exactly two fixed
points of $\rho(\gamma)$ on $\cU$.

$-$ The isometry on $R$ cannot be parabolic, because $M$ quotient
the Seifert fibration is a closed orbifold surface.

\vskip .1in
\noindent
2) Suppose that $M$ is atoroidal. Then the results of \cite{Fe1}
imply that there is a pseudo-Anosov flow $\Phi$ transverse to $\cF$
and regulating for $\cF$. The action of elements of $\pi_1(M)$
on $\cU$ was determined by Proposition 5.3 of \cite{Fe2}.
If $\gamma$ fixes 3 or more points of $\cU$ then $\gamma$
is associated with a periodic orbit of $\Phi$ and $\rho(\gamma)$
has a finite even number of fixed points on $\cU$, which
are alternatively attracting and repelling. If $\gamma$ is
not associated with a periodic orbit of $\Phi$, then $\rho(\gamma)$
has exactly two fixed points on $\cU$, one attracting, one
repelling. 
Finally it could be that $\gamma$ is associated
with an orbit of $\Phi$, but permutes the local prongs.
In this case $\rho(\gamma)$ acts freely on $\cU$, but a power
of $\rho(\gamma)$ fixes at least 4 point on $\cU$.

\vskip .1in
\noindent
3) Finally suppose that the JSJ decomposition of $M$ is
non trivial. As in section \ref{sec.background}
we put the JSJ tori and Klein bottles
in good position with respect to the foliation $\cF$. 
Let $\cT$ be the JSJ tree of $M$: the vertices are
lifts of pieces of the JSJ decomposition
of $M$, an edge is a lift of a torus or Klein bottle of the JSJ
decomposition which connects two lifts of pieces
of the JSJ decomposition.
This tree has a more or less canonical 
embedding into any leaf $F$ of $\wcF$ (modulo moving vertices
in complementary regions, and isotoping edges) preserving the
ordering $-$ see details in section 4 of \cite{FP1}.
The universal circle is canonically homeomorphic to a quotient
of the set of ends of this tree, see \cite{FP1}.

There are many possibilities. 

3.A)  If $\gamma$ acts freely on the tree $\cT$, then $\gamma$ does not
fix any lift of a piece. Then $\rho(\gamma)$ has exactly two fixed
points on $\cU$.

3.B) Suppose that $\gamma$ fixes a lift $\widetilde P$ of a
Seifert piece $P$. 
As in case 1) above $\gamma$ induces an isometry of the quotient
of $\widetilde P$ by the lift of the Seifert fibration, which is
an isometry of a surface embedded in the hyperbolic plane,
but with infinitely many geodesic boundaries. As in case 1),
$\gamma$ could be elliptic, with $\rho(\gamma)$
either not fixing any point on $\cU$
or $\rho(\gamma)$ fixing at least a Cantor set of points in $\cU$,
or exactly 2 points on $\cU$ (if $\rho(\gamma)$ reverses orientation).
See similar analysis in the Appendix of \cite{BFFP}.
If the action of $\gamma$ on the quotient
surface is hyperbolic then $\rho(\gamma)$ 
fixes exactly 2 points on $\cU$.

3.C) Finally if $P$ is atoroidal, look at the action
of $\gamma$ on the leaf spaces of the lifts of the blow ups
of $\cE^s, \cE^u$
as constructed in section \ref{sec.background}.
If these actions on the leaf spaces
are free, then $\gamma$ acts as a translation
on these leaf spaces and $\rho(\gamma)$ fixes exactly two points
on $\cU$, one attracting, one repelling.
If $\gamma$ fixes some leaf of the blow up of
$\cE^s$ then it is associated
with a periodic orbit of the blown up one prong pseudo-Anosov
flow in $P$. If the orbit is in the interior of $P$ (not
peripheral), then a power of $\rho(\gamma)$ fixes finitely many
($\geq 4$) points in $\cU$, which are alternatively attracting
and repelling. If the associated orbit is peripheral, then
$\rho(\gamma)$ fixes infinitely many points on $\cU$.
This is Proposition \ref{prop.superattracting}.


\begin{thebibliography}{2}


\bibitem[BFFP]{BFFP} T. Barthelme, S. Fenley, S. Frankel, R. Potrie,
	{\em Dynamical incoherence for a large class of partially hyperbolic diffeomorphisms}, arxiv.math.2002.10315, to appear in Erg. Th. Dyn. Sys.
	




\bibitem[Be-Pe]{Be-Pe} R. Benedetti and C. Petronio, {\em Lectures on hyperbolic
	geometry}, Universitext, Springer, 1992.

\bibitem[Bl-Ca]{Bl-Ca} S. Bleiler and A. Casson, {\em Automorphisms of surfaces after Nielsen and Thurston}, Cambridge Univ. Press, 1988.
	

\bibitem[Br]{Br} M. Brittenham, {\em Essential laminations in Seifert
	fibered spaces}, Topology {\bf 32} (1993) 61-85.

\bibitem[Bu-Iv]{Bu-Iv} D. Burago and S. Ivanov, {\em Partially hyperbolic diffeomorphisms of $3$-manifolds with abelian fundamental group}, J. Mod. Dyn. {\bf 2} (2008) 541-580.

\bibitem[Cal1]{Cal1} D. Calegari, {\em The geometry of $\rrrr$-covered foliations}, Geometry and Topology {\bf 4} (2000) 457-515.
	
\bibitem[Cal2]{Cal2} D. Calegari, {\em Foliations and the geometry of $3$-manifolds}, Oxford University Press, (2007).

\bibitem[Ca-Du]{Ca-Du} D. Calegari, N. Dunfield, {\em Laminations and groups of homeomorphisms of the circle}, Inven. Math. {\bf 152} (2003) 149-204.

\bibitem[Cand]{Cand} A. Candel, {\em Uniformization of surface laminations}, Ann. Sci. \'{E}cole Norm. Sup. {\bf 26} (1993) 489-516.

\bibitem[Ca-Co]{Ca-Co} A. Candel, L. Conlon, {\em Foliations I and II}, Graduate Studies in Mathematics, 60. American Mathematical Society, Providence, RI, (2000) xiv.402pp and (2003) xiv.545pp.

\bibitem[De-Sh]{De-Sh} P. Dehornoy and M. Shannon, {\em Almost equivalence of algebraic Anosov flows}, math.arxiv.1910.08457.


\bibitem[Fe1]{Fe1} S. Fenley, {\em Foliations, topology and geometry of 
$3$-manifolds: $\rrrr$-covered foliations and transverse pseudo-Anosov flows}, Comm. Math. Helv. {\bf 77} (2002) 415-490.

\bibitem[Fe2]{Fe2} S. Fenley, {\em Rigidity of pseudo-Anosov flows transverse to $\rrrr$-covered foliations}, Comm. Math. Helv. {\bf 88} (2013) 643-676.

\bibitem[FP1]{FP1} S. Fenley and R. Potrie, {\em Mimimality of the action on the universal circle of uniform foliations}, arxiv.math.2001.05522.

\bibitem[FP2]{FP2} S. Fenley and R. Potrie, {\em Partial hyperbolicity and pseudo-Anosov dynamics}, in preparation.





\bibitem[Fr]{Fr} D. Fried, {\em Transitive Anosov flows and pseudo-Anosov
	maps}, Topology {\bf 22} (1983) 299-303.


\bibitem[Ga-Oe]{Ga-Oe} D. Gabai and U. Oertel, {\em Essential laminations and
	$3$-manifolds}, Ann. of Math. {\bf 130} (1989) 41-73.

\bibitem[Gr]{Gr} M. Gromov, {\em Hyperbolic groups}, in Essays in group theory, 75--263, Math. Sci. Res. Inst. Publ. {\bf 8} (1987) Springer, New York.


\bibitem[He]{He} J. Hempel, {\em 3-manifolds}, Ann. of Math. Studies {\bf 86},
	Princeton Univ. Press, 1976.
	
\bibitem[Ja]{Ja} W. Jaco, {\em Lectures on three-manifold topology}, C.B.M.S.
	from A.M.S. {\bf 43}, 1980.

\bibitem[Ja-Sh]{Ja-Sh} W. Jaco and P. Shalen, {\em Seifert fibered spaces 
	in 3-manifolds}, Memoirs A.M.S. {\bf 220}, 1979.

\bibitem[Ka-Le]{Ka-Le} M. Kapovich, B. Leeb, {\em $3$-manifold groups and nonpositive curvature}, Geom. Funct. Anal. {\bf 8} (1998) 841-852.

\bibitem[Mos1]{Mos1} L. Mosher, {\em Dynamical systems and the homology norm of a $3$-manifold I. Efficient intersection of surfaces and flows}, Duke Math. Jour. {\bf 65} (1992) 449-500.

\bibitem[Mos2]{Mos2} L. Mosher, {\em Laminations and flows transverse to finite depth foliations, Part I: Branched surfaces and dynamics}, available from http://newark.rutgers.edu:80/mosher/, Part II in preparation.


\bibitem[Ng]{Ng} H. T. Nguyen, {\em Distortion of surfaces in $3$-manifolds}, J. of Topol. {\bf 12} (2019) 1115-1145.

\bibitem[No]{No} S.P.Novikov, {\em Topology of foliations}, Trans. Moscow Math. Soc., {\bf 14} (1963) 268-305.


\bibitem[Th1]{Th1}  W. Thurston, {\em The geometry and topology of 3-manifolds},
	Princeton University Lecture Notes, 1982.

\bibitem[Th2]{Th2}  W. Thurston, {\em Three dimensional manifolds, Kleinian groups,
	and hyperbolic geometry}, Bull. A.M.S. {\bf 6} (1982) 357-381.

\bibitem[Th3]{Th3}  W. Thurston, {\em A norm for the homology of
three manifolds}, Mem. Amer. Math. Soc. {\bf 59} (1986) No. 339, pp. 99-130.

\bibitem[Th4]{Th4}  W. Thurston, {\em On the geometry and dynamics of diffeomorphisms of surfaces}, Bull. Amer. Math. Soc. {\bf 19} (1998) 417-431.
	
\bibitem[Th5]{Th5}  W. Thurston, {\em Three manifolds, foliations and circles I} arxiv.math.9712268.

\bibitem[Th6]{Th6} W. Thurston, {\em Three manifolds, foliations and circles, II. Transverse asymptotic geometry of foliations}, preprint.

\bibitem[Tic]{Tic} D. Tischler, {\em On fibering certain foliated manifolds over $\mathbb S^1$}, Topology {\bf 9} (1970) 153-154.

\end{thebibliography}
\end{document}